\documentclass[10pt]{article}

\usepackage{graphicx}
\usepackage{amssymb}
\usepackage{amsmath}
\usepackage{amsthm}

\newtheorem{df}{Definition}[section]
\newtheorem{definition}[df]{Definition}
\newtheorem{theorem}[df]{Theorem}
\newtheorem{proposition}[df]{Proposition}
\newtheorem{lemma}[df]{Lemma}
\newtheorem{conjecture}[df]{Conjecture}
\newtheorem{procedure}[df]{Procedure}
\newtheorem{remark}[df]{Remark}

\newcounter{listacnt}\renewcommand{\thelistacnt}{\alph{listacnt}}
\newenvironment{theoremlist}{\begin{list}{\textnormal{(\thelistacnt)}}%
{\settowidth{\labelwidth}{(b)}%
\setlength{\topsep}{1.5pt plus 0.5pt minus 0.5pt}%
\setlength{\itemsep}{1pt plus 0.2pt minus 0.2pt}%
\setlength{\parsep}{0.1pt plus 0.1pt minus 0.1pt}%
\usecounter{listacnt}}}{\end{list}}

\newcommand{\ba}{\bar{a}}
\newcommand{\da}{\dot{a}}
\newcommand{\db}{\dot{b}}
\newcommand{\bb}{\bar{b}}
\newcommand{\rd}{{\rm d}}
\newcommand{\trd}{\tilde{\rm d}}
\newcommand{\bx}{\bar{x}}
\newcommand{\dx}{\dot{x}}
\newcommand{\tx}{\tilde{x}}
\newcommand{\bL}{\bar{L}}
\newcommand{\dL}{\dot{L}}
\newcommand{\C}{\mathcal{C}}
\newcommand{\cI}{\mathcal{I}}
\newcommand{\dal}{\Delta_\alpha}
\newcommand{\lec}{\le_{cw}}
\newcommand{\eps}{\varepsilon}
\newcommand{\gc}{>_{cw}}
\newcommand{\bydef}{\,\stackrel{\mbox{\tiny\textnormal{\raisebox{0ex}[0ex][0ex]{def}}}}{=}\,} 
\newcommand{\dagA}{A^{\dagger}}
\newcommand{\zft}{\zf^T}
\newcommand{\zf}{{0_{_F}}}
\newcommand{\cF}{\mathcal{F}}

\title{Recent advances about the uniqueness of the slowly oscillating periodic solutions of Wright's equation}

\date{}

\author{Jean-Philippe Lessard\thanks{Department of Mathematics, Rutgers University, 110 Frelinghusen Rd, Piscataway, NJ  08854-8019, USA ({\tt lessard@math.rutgers.edu}). }
\thanks{Department of Mathematics, VU University Amsterdam, De Boelelaan 1081, The Netherlands ({\tt jlessard@few.vu.nl})} }

\begin{document}

\maketitle

\begin{abstract}
An old conjecture in delay equations states that Wright's equation \[ y'(t)= - \alpha y(t-1) [ 1+y(t)],~~ \alpha \in \mathbb{R} \]  has a unique slowly oscillating periodic solution (SOPS) for every parameter value $\alpha>\pi/2$. We reformulate this conjecture and we use a method called validated continuation to rigorously compute a global continuous branch of SOPS of Wright's equation. Using this method, we show that a part of this branch does not have any fold point, partially answering the new reformulated conjecture. 
\end{abstract}
 
\section{Introduction} \label{sec:introduction}

In 1955,  Edward M. Wright considered the equation
\begin{equation} \label{eq:wright}
y'(t)= - \alpha y(t-1) [ 1+y(t)], ~\alpha>0,
\end{equation}
because of its role in probability methods applied to the theory of distribution of prime numbers, and he proved the existence of bounded non constant solutions which do not tend to zero, for every $\alpha>\pi/2$ \cite{wright}. Throughout this paper, we refer to equation (\ref{eq:wright}) as Wright's equation. Since the work presented in \cite{wright}, equation (\ref{eq:wright}) has been studied by many mathematicians (e.g. see \cite{C-MP1,Jones2, Jones1, kakutani-markus, kaplan-yorke1, kaplan-yorke2,nussbaum4, nussbaum2, nussbaum5}). In 1962, G.S. Jones proved the existence of periodic solutions of (\ref{eq:wright}) for $\alpha>\pi/2$ \cite{Jones2}. Then in \cite{Jones1}, he studied their quantitative properties and he made the following remark.
\begin{quote} {\em The most important observable phenomenon resulting from these numerical experiments is the apparently rapid convergence of solutions of (\ref{eq:wright}) to a {\bf single} cycle fixed periodic form which seems to be independent of the initial specification on $[-1,0]$ to within translations.}
\end{quote}
The cycle fixed periodic form he refers to is a slowly oscillating periodic solution.
\begin{definition} {\em A slowly oscillating periodic solution ({\em SOPS}) of (\ref{eq:wright}) is a periodic solution $y(t)$ with the following property: there exist $q>1$ and $p>q+1$ such that, up to a time translation, $y(t)>0$ on $(0,q)$, $y(t)<0$ on $(q,p)$, and $y(t+p)=y(t)$ for all $t$ so that $p$ is the minimal period of $y(t)$. }
\end{definition}

After Jones made the above remark, the question of the uniqueness of SOPS in (\ref{eq:wright}) became popular and is still under investigation after almost fifty years.
\begin{conjecture} \label{conj:uniqueSOPS}
For every $\alpha > \frac{\pi}{2}$, (\ref{eq:wright}) has a unique SOPS.
\end{conjecture}

It is worth mentioning that if Conjecture~\ref{conj:uniqueSOPS} is true, then the unique SOPS  
attracts a dense and open subset of the phase space (e.g. see \cite{MP_Walther}). Let us reformulate Conjecture~\ref{conj:uniqueSOPS}, considering the partial work that was done since Jones's comment in \cite{Jones1}. In 1977, Chow and Mallet-Paret showed that there is a supercritical (forward in $\alpha$) Hopf bifurcation of SOPS from the trivial solution at $\alpha=\pi/2$ \cite{C-MP1}. We denote this branch of SOPS by $\cF_0$. In 1989, Regala proved a result that implies that there cannot be any secondary bifurcation from $\cF_0$ \cite{regala}. Hence, $\cF_0$ is a regular curve in the $(\alpha,y)$ space. In 1991, Xie used asymptotic estimates for large $\alpha$ to prove that for $\alpha> 5.67$, (\ref{eq:wright}) has a unique SOPS up to a time translation  \cite{xie_thesis,xie1}. 
Here is a remark he made after he stated his result on p. 97 of his thesis \cite{xie_thesis}.
\begin{quote} {\em The result here may be further sharpened. However, [$\ldots$] the arguments here can not be used to prove the uniqueness result for SOPS of (\ref{eq:wright}) when $\alpha$ is close to $\frac{\pi}{2}$. }
\end{quote}

Hence, his method might help to decrease the value $5.67$, but new mathematical ideas are required to solve Conjecture~\ref{conj:uniqueSOPS}. Based on the above discussion, here is a reformulation of the remaining parts of the conjecture.

\begin{conjecture} \label{conj:new_uniqueSOPS}
Denote by $\cF_0$ the branch of SOPS that bifurcates (forward in $\alpha$) at $\pi/2$. Then 
\begin{enumerate}
\item $\cF_0$ does not have any fold in $\alpha \in (\frac{\pi}{2},5.67]$;
\item there are no connected components (isolas) of SOPS in $\alpha \in (\frac{\pi}{2},5.67]$.
\end{enumerate}
\end{conjecture}

\begin{figure}[ht!]
\centering
\includegraphics[width=8cm]{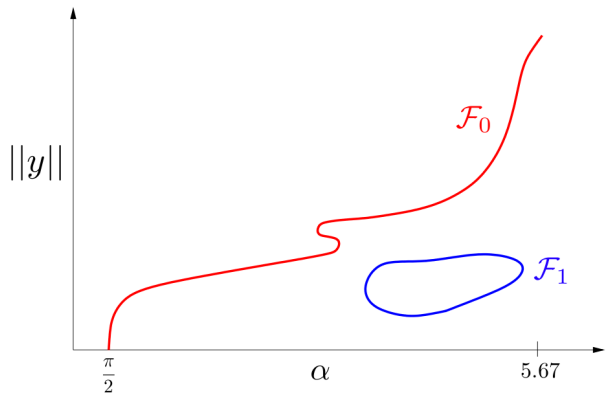}
\caption{Conjecture~\ref{conj:new_uniqueSOPS} fails if in the parameter range corresponding to $\alpha \in (\frac{\pi}{2},5.67]$, there exists a fold on $\cF_0$ or there exists an isola $\cF_1$ of SOPS.}
\label{fig:isolas}
\end{figure}

In this paper, we propose to use a method called {\em validated continuation} in the parameter $\alpha$ to partially prove the first part of Conjecture~\ref{conj:new_uniqueSOPS}. This method was originally introduced in \cite{DLM} as a computationally efficient tool to compute equilibrium solutions of partial differential equations (PDEs) with polynomial nonlinearities. It was then adapted to compute equilibria of PDEs for large (discrete) range of parameter values \cite{GLM}. Afterward, it was combined with variational methods and tools from algebraic topology to prove the existence of chaos for a class of fourth order nonlinear ordinary differential equations \cite{BL}. In \cite{BLM}, validated continuation was generalized to compute global smooth branches of solution curves of differential equations, both in the context of parameter and pseudo-arclength continuation. Finally, in a forthcoming work, the method is adjusted to compute equilibria of high dimensional PDEs \cite{GL}. In this paper, we use the theory of validated continuation developed in \cite{BLM} to compute a global continuous curve of SOPS of Wright's equation.

\begin{theorem} \label{thm:SOPS}
Let $\varepsilon=7.3165 \times 10^{-4}$. Then the part of $\cF_0$ corresponding to the parameter range $\alpha \in \left[ \frac{\pi}{2}+\varepsilon,2.3 \right]$ does not have any fold.
\end{theorem} 

\begin{figure}[ht!]
\centering
\includegraphics[width=8cm]{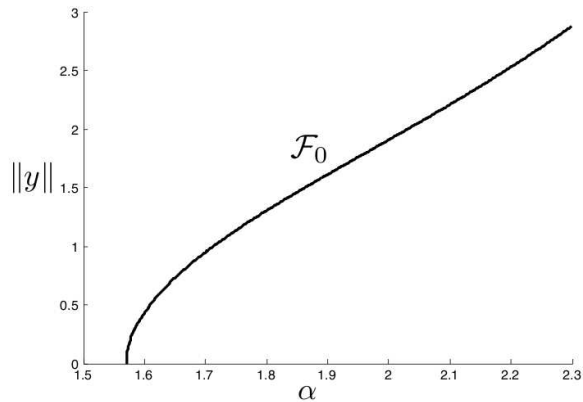}
\caption{Geometric representation of the result of Theorem~\ref{thm:SOPS}. This curve represents a rigorous computation of a section of the set $\cF_0$. On the picture, the vertical axis is given by $\| y \| = \sup \left\{ |y(t)|~;~{t \in [0,p],~{\rm where}~  p~ {\rm is~the~period~of~} y}\right\}$.
}
\label{fig:uniqueSOPS}
\end{figure}
For a geometric representation of Theorem~\ref{thm:SOPS}, we refer to Figure~\ref{fig:uniqueSOPS}. Before going into the details of the proof, let us make a few comments on the statement of Theorem~\ref{thm:SOPS}. The reason why the result is valid only up to $\alpha=2.3$ does not have any theoretical justification. This is purely computational. In fact, when $\alpha$ grows, the proof becomes computationally difficult mainly because of the following facts. First of all, our computer-assited proof requires the computation of several sums which we compute using iterative loops with the {\em Matlab} interval arithmetic package {\em Intlab} \cite{Intlab} which is slow to evaluate loops of large size. A second observation is that the step size $\dal$ in the parameter $\alpha$ decreases significantly when one increases the parameter $\alpha$. Hence, for larger $\alpha$, the rigorous continuation still runs, but the step size decreases significantly. We come back to these issues in Section~\ref{sec:conclusion}, where we make suggestions on how to possibly improve the result of Theorem~\ref{thm:SOPS}. 

Another comment regarding Theorem~\ref{thm:SOPS} is that validated continuation in $\alpha$ cannot help ruling out the existence of a fold in the parameter range $\alpha \in ~ ] \pi/2,\pi/2+\eps[$. This is due to the fact that the method requires having contractions which are uniform in the parameter $\alpha$. Because the trivial periodic solution $y=0$ is non hyperbolic at $\alpha=\pi/2$, the uniform contraction in the parameter $\alpha$ fails to exist near $\alpha=\pi/2$. That raises the following question: How can we make sure that the global branch of SOPS obtained with validated continuation for $\alpha \in [\pi/2+\eps,2.3]$ actually comes from the Hopf bifurcation at $\alpha=\pi/2$? It turns out that we can {\em regularize} the problem at $\alpha=\pi/2$ with the change of variable $y(t)=\beta z(t)$ and obtain a new problem (with continuation parameter $\beta \ge 0$) having a non trivial hyperbolic periodic solution $z(t)$ at $\beta=0$ and $\alpha=\pi/2$. This new problem, having now $\alpha$ as a variable (as opposed to a parameter), can be studied with validated continuation again, since uniform contractions can be proved to exist near $\beta=0$ and $\alpha=\pi/2$. This is done in Section~\ref{sec:2nd_part}, where a rigorous continuation in the new parameter $\beta \ge 0$ is performed in order to show that the branch of SOPS that we computed in the parameter interval $\alpha \in [\pi/2+\eps,2.3]$ is in fact the one that bifurcates from the trivial solution at $\alpha=\pi/2$. 

Finally, it is important to mention that the value of $\eps$ can be made smaller using our method. The choice of $\eps=7.3165 \times 10^{-4}$ is made arbitrarily and we believe that with significant extra computational effort, this value can be pushed down up to $\eps=1 \times 10^{-8}$. Once again, we discuss this possible improvement in Section~\ref{sec:conclusion}.

The paper is organized as follows.  In Section~\ref{sec:f_set_up}, we transform the study of periodic solutions of (\ref{eq:wright}) into the study of the solutions of a parameter dependent infinite dimensional problem $f(x,\alpha)=0$. In Section~\ref{sec:fs_fpe}, the problem $f(x,\alpha)=0$ is modified into an equivalent fixed point problem $T(x,\alpha)=x$, whose fixed points correspond to zeros of $f$. The equivalence of the problem is shown and the functional analysis setting is introduced. In Section~\ref{sec:validated_continuation}, we introduce the validated continuation method in the fashion of \cite{BLM}. In Section~\ref{sec:proof}, we prove Theorem~\ref{thm:SOPS} and finally, we conclude with possible improvements in Section~\ref{sec:conclusion}. The computer programs used to assist the proof of  Theorem~\ref{thm:SOPS} can be found at \cite{code}.

%
%
%
%
%
%%%%%%%%%%%%%%%%%%%%%%
%%%%%%%%%%%%%%%%%%%%%%
%%							        %%
%%		SECTION f(x,\alpha)=0	        %%
%%							        %%
%%%%%%%%%%%%%%%%%%%%%%
%%%%%%%%%%%%%%%%%%%%%%
%
%
%
%
%
\section{Set up of the problem \boldmath$f(x,\alpha)=0$\unboldmath} \label{sec:f_set_up}
The goal of this section is to transform the problem of looking for periodic solutions $y(t+p)=y(t)$ of (\ref{eq:wright}) into the study of the solutions of a parameter dependent infinite dimensional problem $f(x,\alpha)=0$. Let us introduce $L$ to be the a priori unknown frequency of the periodic solution $y$. In other words, $p=\frac{2 \pi}{L}$. Hence, consider the following expansion of the periodic solution $y$ in Fourier series
\begin{equation} \label{eq:complex_fourier}
y(t)=\sum_{k=-\infty}^{\infty} c_k e^{ikLt},
\end{equation}
where the $c_k$ are complex numbers satisfying $c_{-k}=\overline{c_k}$. This is due to the fact that $y \in \mathbb{R}$. Plugging the two expressions
\[ y(t-1)=\sum_{k=-\infty}^{\infty} c_k e^{-ikL}e^{ikLt} ~~{\rm and}~~ y'(t)=\sum_{k=-\infty}^{\infty} c_k ikL e^{ikLt} \]
in (\ref{eq:wright}) and putting all terms on one side of the equality, one gets a new problem to solve for, namely
\[
\sum_{k_0=-\infty}^{\infty}  \left[ ik_0L+ \alpha e^{-ik_0L} \right] c_{k_0} e^{ik_0Lt} + \alpha \left[ \sum_{k_1=-\infty}^{\infty} c_{k_1} e^{-ik_1L} e^{ik_1Lt} \right]
\left[ \sum_{k_2=-\infty}^{\infty} c_{k_2} e^{ik_2Lt} \right] =0.
\]
The left hand side of this last equation being a periodic solution with period $\frac{2 \pi}{L}$, one computes its Fourier coefficients by taking the inner product with $e^{ikLt}$, for $k \in \mathbb{Z}$. This procedure leads to the following countable system of equations
\begin{equation} \label{eq:g}
g_k \bydef\left[ ikL + \alpha e^{-ikL} \right] c_k + \alpha
\sum_{k_1+k_2=k} e^{-ik_1 L} c_{k_1} c_{k_2} =0 ,~~ k \in \mathbb{Z}.
\end{equation}
Since $c_{-k}=\overline{c_k}$ implies that $g_{-k}=\overline{g_k}$, we only need to consider the cases $k \ge 0$ when solving for (\ref{eq:g}).  Note that the frequency $L$ of $y$ being unknown, we leave it variable and we are going to solve for it when solving $f=0$. Denoting the real and the imaginary part of $c_k$ respectively by $a_k$ and $b_k$, an equivalent expansion for (\ref{eq:complex_fourier}) is given by
\begin{equation} \label{eq:real_fourier}
y(t)=a_0+2 \sum_{k=1}^{\infty} \left[ a_k \cos{kLt}-b_k \sin{kLt}
\right] . \end{equation}
Note that $a_k=a_{-k}$ and $b_k=-b_{-k}$. Hence, we get that $b_0=0$. Let 
\[
x_k \bydef \left\{ \begin{array}{ccc}  (L,a_0) ,~k=0 \\ 
(a_k,b_k) ,~ k >0 \end{array} \right.
\]
and $x \bydef (x_0,x_1,\cdots,x_k,\cdots)^T$.  Let us denote by $x_{k,1}$ and $x_{k,2}$ the first and the second component of $x_k \in \mathbb{R}^2$, respectively. In order to eliminate arbitrary shifts, we impose the normalizing condition $y(0)=a_0+2 \sum_{k=1}^{\infty} a_k=0$. Hence, let us introduce the following function $h$, which will ensure, by solving $h=0$, that the scaling condition $y(0)=0$ is satisfied:
\[
h(x) \bydef a_0+2 \sum_{k=1}^{\infty} a_k .
\]
For $k\ge 0$, consider the real and the imaginary parts of $g_k$, given respectively by
\begin{eqnarray} \label{eq:fk1}
Re(g_k) (x,\alpha) &=& ( \alpha \cos{kL} ) a_k + (-kL+\alpha \sin{kL}) b_k \label{eq:f_k1}
\\
&& \nonumber + \alpha
\sum_{k_1+k_2=k}(\cos{k_1 L}) (a_{k_1}a_{k_2}-b_{k_1}b_{k_2}) +
(\sin{k_1 L}) (a_{k_1}b_{k_2}+b_{k_1}a_{k_2}), \\
Im(g_k) (x,\alpha)  &=& \label{eq:f_k2}
-(-kL+\alpha \sin{kL}) a_k + (\alpha \cos{kL}) b_k \label{eq:fk2}
\\ && + \alpha \nonumber
\sum_{k_1+k_2=k}-(\sin{k_1 L}) (a_{k_1}a_{k_2}-b_{k_1}b_{k_2}) +
(\cos{k_1 L}) (a_{k_1}b_{k_2}+b_{k_1}a_{k_2}). 
\end{eqnarray}
Note that $g_{-k}=\overline{g_k}$ implies that $Im(g_0) = 0$. Hence, we do not incorporate $Im(g_0)$ in the formulation of $f$. Hence, the function $f$ is defined component-wise by
\[
f_k(x,\alpha)= \left\{
\begin{array}{ll} 
\left( 
\begin{array}{cc}
h(x)  \\ Re(g_0)(x,\alpha)
\end{array}
\right),~k=0
\\
\left( 
\begin{array}{cc}
Re(g_k)(x,\alpha)  \\ Im(g_k)(x,\alpha)
\end{array}
\right),~k>0
\end{array} \right. 
\]

Consider the notation $f_{k,1}$ (resp. $f_{k,2}$) to denote the first (resp. second) component of $f_k \in \mathbb{R}^2$. Defining $f=\{f_k\}_{k \ge 0}$, we show in Section~\ref{sec:fs_fpe} that finding periodic solution $y(t)$ of (\ref{eq:wright}) satisfying $y(0)=0$ is equivalent to finding solutions of the infinite dimensional parameter dependent problem
\begin{equation} \label{eq:f}
f(x,\alpha)=0.
\end{equation}
%
%
%
%
%
%
%%%%%%%%%%%%%%%%%%%%%%
%%%%%%%%%%%%%%%%%%%%%%
%%							        %%
%%		SECTION T(x,\alpha)=x	        %%
%%							        %%
%%%%%%%%%%%%%%%%%%%%%%
%%%%%%%%%%%%%%%%%%%%%%
%
%
%
%
%
\section{Set up of the fixed point equation \boldmath$T(x,\alpha)=x$\unboldmath ~and functional analysis setting} \label{sec:fs_fpe}

The purpose of this section is to transform the problem $f(x,\alpha)=0$ into a fixed point equation $T(x,\alpha)=x$. Then, the idea will be to apply an uniform contraction mapping argument on $T$. Let us first put ourself in a functional analysis setting by introducing a Banach space which is convenient for our study. The key ingredient in defining the space is that periodic solutions of Wright's equation are $C^\infty$ \cite{nussbaum3}. This implies that the Fourier coefficients of the expansion (\ref{eq:real_fourier}) goes to zero faster than any algebraic decay. For $s >0$, consider the weights
\begin{equation}\label{e:weights}
  \omega_k = \left\{ 
  \begin{array}{ll}
 ~ 1, &  k =0; \\
  |k|^s, & k \ne 0. \\
  \end{array}
 \right.
\end{equation}
These weights are used to define the norm
\begin{equation}\label{e:norm}
  \|x\|_s \bydef \sup_{k=0,1,\dots} |x_k|_\infty \omega_k ,
\end{equation}
where $|x_k|_\infty=\max \{ |x_{k,1}|,|x_{k,2}|\}$, and the sequence space
$$
  \Omega^s = \{ x=(x_0,x_1,x_2,\dots) \, , \, \| x \|_s < \infty \},
$$
consisting of sequences with algebraically decaying tails. Since the Fourier coefficients $\{x_k\}_{k\ge 0}$ decay faster than any given power of $k$, the set $\Omega^s$ contains all sequences $(L,a_0,a_1,b_1,\dots)$ obtained from the Fourier expansion (\ref{eq:real_fourier}) of any periodic solutions of (\ref{eq:wright}). We are ready to define the fixed point operator $T$.

First of all, note that $T$ will partially be constructed with the help of the computer. For that matter, we then truncate the infinite dimensional problem (\ref{eq:f}) into a finite dimensional one. More precisely, consider the finite dimensional projection $f^{(m)}:\mathbb{R}^{2m} \times \mathbb{R} \rightarrow \mathbb{R}^{2m}$ defined component-wise by
\begin{equation} \label{eq:f_Galerkin}
f_k^{(m)}(x_0,\dots,x_{m-1},\alpha)  \bydef f_k \left( (x_0,\dots,x_{m-1},0_\infty),\alpha \right) , ~k=0,\dots,m-1,
\end{equation}
where $0_\infty=(0)_{j \ge 0}$. Consider a parameter value $\alpha_0>\pi/2$. Recall from the discussion in Section~\ref{sec:introduction} that since we aim for a contraction mapping argument, we consider only parameter values $\alpha_0>\pi/2$. Indeed, at $\alpha_0=\pi/2$, the trivial solution is non hyperbolic, meaning that $D_xf(0,\pi/2)$ is not injective. Suppose that at $\alpha_0$, we computed numerically $\bx \in \mathbb{R}^{2m}$ such that 
\begin{equation} \label{eq:approx_zero}
f^{(m)}(\bx,\alpha_0) \approx 0. 
\end{equation}
This is done with a Newton-like iterative scheme. To simplify the presentation, we identify $\bx=(\bL,\ba_0,\ba_1,\bb_1,\dots,\ba_{m-1},\bb_{m-1})^T$ with $(\bx,0_\infty)$. Define 
\[
\Lambda_k \bydef \frac{\partial f_k}{\partial x_k} (\bx,\alpha_0)=
\left(
\begin{array}{cccc} 
\frac{\partial f_{k,1}}{\partial x_{k,1}} (\bx,\alpha_0) & \frac{\partial f_{k,1}}{\partial x_{k,2}} (\bx,\alpha_0) \\
\frac{\partial f_{k,2}}{\partial x_{k,1}} (\bx,\alpha_0) & \frac{\partial f_{k,2}}{\partial x_{k,2}} (\bx,\alpha_0)
\end{array}
\right).
\]
We use the subscript $(\cdot)_{_F}$ to denote the $2(2m-1)$ entries corresponding to $k=0,\cdots,2m-2$.  Let $J_{_F}$ be a numerical approximation of the inverse of  $D_xf^{(2m-1)}(\bx,\alpha_0)$, $0_2$ be the $2 \times 2$ zero matrix and let $\zf$ be the $2 \times 2(2m-1)$ zero matrix. Let
\begin{equation} \label{eq:A}
A \bydef \left[\begin{array}{cccccc}
J_{_F}&\,& \zft & \zft & \zft & \cdots \\[2mm]
\zf &\,& \Lambda^{-1}_{2m-1} & 0_2 &0_2 &\cdots \\
\zf &\,& 0_2 & \Lambda^{-1}_{2m} & 0_2 & \cdots \\
\zf &\,& 0_2 & 0_2 & \Lambda^{-1}_{2m+1} &  \\
\vdots         &\,& \vdots & \vdots & & \ddots  
        \end{array}\right] \, ,
\end{equation}
which acts as an \emph{approximate inverse} of the linear operator $D_xf(\bx,\alpha_0)$. More precisely, given $x \in \Omega^s$, one has that 
\begin{equation} \label{eq:def_A}
Ax=\left(J_{_F}x_{_F}, \Lambda^{-1}_{2m-1} x_{2m-1},\Lambda^{-1}_{2m} x_{2m},\dots \right).
\end{equation}
\begin{lemma} Given (\ref{eq:A}) and (\ref{eq:def_A}), we have that $A: \Omega^s \rightarrow \Omega^{s+1}$.
\end{lemma}
\begin{proof}
First of all, there exists a constant $2 \times 2$ matrix $\Xi$ such that
\[
\left| {\Lambda_k}^{-1} \right| \lec \frac{1}{k} \Xi,
\]
for all $k \ge 2m-1$ (see Lemma~\ref{lm:Xi_M}), where $| \cdot |$ means component-wise absolute values and $\lec$ means component-wise inequalities. Considering $x \in \Omega^s$, one gets that
\begin{eqnarray*}
\|Ax\|_{s+1} &=& \max \left\{ |(Ax)_0|_\infty , \max_{k=1,\dots,2m-2} |(Ax)_k|_\infty k^{s+1},\sup_{k \ge 2m-1 } |(Ax)_k|_\infty k^{s+1}  \right\}  \\
&=&  \max \left\{ |(J_{_F}x_{_F})_0|_\infty , \max_{k=1,\dots,2m-2} |(J_{_F}x_{_F})_k|_\infty k^{s+1},\sup_{k \ge 2m-1 } | \Lambda_k^{-1}x_k |_\infty k^{s+1}  \right\} \\
&\le& \max \left\{ |(J_{_F}x_{_F})_0|_\infty , \max_{k=1,\dots,2m-2} |(J_{_F}x_{_F})_k|_\infty k^{s+1},\sup_{k \ge 2m-1 } |  \Xi x_k |_\infty k^s  \right\}\\
& <& \infty,
\end{eqnarray*}
because $\|x\|_s=\sup_{k\ge 0} | x_k |_\infty \omega_k < \infty$ and $\Xi$ is a constant matrix.
\end{proof}
Let us comment on how, in practice, we make sure that the linear operator $A$ is invertible. First of all, we verify that 
\begin{equation} \label{eq:inv_condition1}
\| J_{_F} D_x f^{(2m-1)}(\bx,\alpha_0) -I_{_F} \|_\infty<1,
\end{equation}
with $I_{_F}$ being the $2(2m-1) \times 2(2m-1)$ identity matrix. If such inequality is satisfied, we get that $J_{_F}$ is invertible. Recalling the definitions of $f_{k,1}$ and $f_{k,2}$ given in (\ref{eq:f_k1}) and (\ref{eq:f_k2}), respectively, and considering $k \ge 2m-1$, we get that 
\begin{equation} \label{eq:Lambda_k} 
\Lambda_k = 
\left( \begin{array}{cc} \tau_k & \delta_k \\ -\delta_k & \tau_k \end{array} \right),
\end{equation}
where $\tau_k\bydef\alpha_0 \ba_0 +\alpha_0  (1+\ba_0) \cos{k\bL}$ and $\delta_k\bydef-k\bL+\alpha_0(1+\ba_0) \sin{k\bL}$. Hence, a sufficient condition for $\Lambda_k$ to be invertible for all $k \ge 2m-1$ is that 
\begin{equation} \label{eq:inv_condition2}
m > \frac{1}{2} \left[ \frac{\alpha_0 |1+\ba_0|}{\bL} + 1 \right].
\end{equation}
Indeed, by (\ref{eq:inv_condition2}), we get that $\delta_k<0$ for all $k \ge 2m-1$ and we can conclude that $det(\Lambda_k) = \tau_k^2+\delta_k^2>0$, for all $k \ge 2m-1$. Hence, if conditions (\ref{eq:inv_condition1}) and (\ref{eq:inv_condition2}) hold, the linear operator $A$ defined in (\ref{eq:A}) is invertible. 

Given a parameter value $\alpha \ge \alpha_0$, we define the fixed point operator $T:\Omega^s \times \mathbb{R}$ to $\Omega^{s}$ by 
\begin{equation} \label{eq:T_alpha}
T(x,\alpha)=x-A f(x,\alpha) 
\end{equation}

It is now important to remark that even if we constructed the operator $T$ in a computer-assisted fashion, we still think of it as an abstract object. The finite part is stored on a computer, and the tail part, consisting of the sequence of matrices $\{ \Lambda_k^{-1} \}_{k \ge 2m-1}$, is defined abstractly.

\begin{lemma}\label{l:technical}
We have the following:
\begin{theoremlist}
\item
Let $s_0  \ge 2$ and fix $\alpha$.
Zeros of $f(x,\alpha)$, or, equivalently, 
fixed points of $T(x,\alpha)$, that are in $\Omega^{s_0}$, 
are in $\Omega^{s}$ for all $s \geq s_0$.
\item\label{i:iff}
Let $s \ge 2$. 
A sequence $x=(x_0,x_1,x_2,\dots) \in \Omega^s$ is a zero of $f$, or a fixed
point of $T$, 
\emph{if and only if} $y$ given by (\ref{eq:real_fourier}) is a periodic
solution of~(\ref{eq:wright}) with $y(0)=0$.
\end{theoremlist}
\end{lemma}

\begin{proof}
For part (a), equivalence of zeros of $f$ and fixed points of $T$ is obvious, since the operator $A$ is invertible. Suppose there exists $x \in \Omega^{s_0}$ such that $f(x,\alpha)=0$. Recalling that $x_k=(a_k,b_k)$ for $k \ge 1$, that $c_k=a_k+i b_k$ and equation (\ref{eq:g}), we get that $g_k=0$, for every $ k \ge 0$. Hence, for all $k \ge 0$, we get that
\begin{equation} \label{eq:lemma_proof}
\left[ ikL + \alpha e^{-ikL} \right] c_k = - \alpha
\sum_{k_1+k_2=k} e^{-ik_1 L} c_{k_1} c_{k_2} .
\end{equation}
However, we have that 
\[
\left| \sum_{k_1+k_2=k} e^{-ik_1 L} c_{k_1} c_{k_2} \right| \le
2 \|x\|_{s_0}^2 \left| \sum_{k_1+k_2=k} \frac{1}{\omega_{k_1} \omega_{k_2}} \right| \le \frac{B}{k^{s_0}},
\]
where $B\ge 0$ is independent of $k$ (see equation (\ref{eq:upper_boundS1_tail0}) in Lemma~\ref{lem:estimates_sum_tail_case}). Combining this inequality with (\ref{eq:lemma_proof}), we get that $k^{s_0+1} c_k$ is uniformly bounded. This implies that $x \in \Omega^{s_0+1}$. Repeating this argument, we can conclude that zeros of $f(x,\alpha)$ that are in $\Omega^{s_0}$, are in $\Omega^{s}$ for all $s \geq s_0$.

Finally, because the tail of a fixed point of $T$ decays faster than any algebraic
rate, all sums may be differentiated term by term, hence $y$ defined by
(\ref{eq:real_fourier}) is a periodic solution of (\ref{eq:wright}) with $y(0)=0$. On the other hand, any periodic solution of (\ref{eq:wright}) is $C^\infty$, hence the tail of its Fourier transform decays faster than any algebraic rate, and thus, by standard arguments, 
the Fourier transform solves $f=0$, and part (b) follows.
\end{proof}

We are now ready to introduce validated continuation.

\section{Validated Continuation} \label{sec:validated_continuation}
Validated continuation \cite{BL, BLM, DLM, GL, GLM} is a rigorous computational method to continue, as we move a parameter, the zeros of infinite dimensional parameter dependent problems. In our context, we use this technique to continue solutions of (\ref{eq:f}), as we move the parameter $\alpha$. Lemma~\ref{l:technical}\ref{i:iff} shows that the problem of finding periodic solutions $y$ of
(\ref{eq:wright}) such that $y(0)=0$ is equivalent to studying fixed points of $T$. 
We will find balls in $\Omega^s$ on which $T$, for fixed $\alpha$, is a contraction mapping,
thus leading to periodic solutions $y$ of~(\ref{eq:wright}) satisfying $y(0)=0$. 

Let $\alpha_0>\pi/2$ considered in Section~\ref{sec:fs_fpe} and suppose that we computed a {\it tangent} $\dx \in \mathbb{R}^{2m}$ such that  
\begin{equation} \label{eq:tangent}
D_xf^{(m)}(\bx,\alpha_0) \dx + \frac{\partial f^{(m)}}{\partial \alpha} (\bx,\alpha_0) \approx 0.
\end{equation}
As in Section~\ref{sec:fs_fpe}, we identify $\dx =(\dL,\da_0,\da_1,\db_1,\dots,\da_{m-1},\db_{m-1})^T$ with $(\dx,0_\infty)$. Let us define the ball of radius $r$ in $\Omega^s$ (with norm $\|\cdot\|_s$) , centered at the origin, 
\begin{equation} \label{eq:Br}
B(r) \bydef \prod_{k=0}^{\infty} \left[ -\frac{r}{\omega_k},\frac{r}{\omega_k} \right]^2
\end{equation}
so that a point $b \in B(r)$ can be factored $b=ur$, with $u \in B(1)$. For $\dal = \alpha-\alpha_0 \ge 0$, we define the {\it predictor based at $\alpha_0$} by
\begin{equation} \label{eq:x_alpha}
x_\alpha=\bx+ \dal \dx 
\end{equation} 
and balls centered at $x_\alpha$
\begin{equation} \label{eq:W_x_alpha}
B_{x_\alpha}(r)= x_\alpha + B(r). 
\end{equation}
\begin{definition}
Let $u,v \in \mathbb{R}^{m \times n}$. We define the component-wise inequality by $\lec$ and say that $u \lec v$ if $u_{i,j} \le v_{i,j}$, for all $i=1,\dots,m$  and $j=1,\dots,n$.
\end{definition}
To show that $T$ is a contraction mapping, we need component-wise positive bounds
$Y_k={\tiny  \left( \hspace{-.2cm} \begin{array}{cc} Y_{k,1} \\ Y_{k,2} \end{array}\hspace{-.2cm} \right)},Z_k = {\tiny  \left( \hspace{-.2cm} \begin{array}{cc} Z_{k,1} \\ Z_{k,2} \end{array}\hspace{-.2cm} \right)} \in\mathbb{R}^2$ for each $k \ge 0 $, such that, with $\dal=\alpha-\alpha_0$,
\begin{equation} \label{eq:Y_k}
  \Bigl| [T(x_\alpha,\alpha)-x_\alpha]_k \Bigr| \lec Y_k(\dal) , 
\end{equation}
and
\begin{equation} \label{eq:Z_k}
 \sup_{b,c \in B(r)} \Bigl| [D_xT(x_\alpha+b,\alpha)c]_k \Bigr|
  \lec Z_k(r,\dal) .
\end{equation}
We will find such bounds in Sections~\ref{sec:Y} and~\ref{sec:Z},
respectively. We only consider $\dal \ge 0$, since we initiate the continuation at the parameter value $\alpha_0=\frac{\pi}{2}+\varepsilon$ and move forward. The proof of the following Lemma can be found in \cite{BL}.

\begin{lemma}\label{lem:ex_uni}
Fix $s\ge 2 $ and $\alpha=\alpha_0+\dal$.
If there exists an $r>0$ such that $\|Y + Z\|_s <r$, 
with $Y=(Y_0,Y_1,\dots)$ and $Z=(Z_0,Z_1,\dots)$ the bounds
as defined in
(\ref{eq:Y_k}) and (\ref{eq:Z_k}),
then there is a unique $\tx_\alpha \in B_{x_\alpha}(r)$ such that $f(\tx_\alpha,\alpha)=0$.
\end{lemma}

In order to verify the hypotheses of Lemma~\ref{lem:ex_uni} in a
computationally efficient way, we introduce the notion of 
\emph{radii polynomials}.  
Namely, as will become clear in Sections~\ref{sec:Y} and~\ref{sec:Z},
the functions $Y_k(\dal)$ and $Z_k(r,\dal)$ are polynomials in their
independent variables. In fact, they are constructed to be monotone increasing in $\dal$. Also, for $k \geq M \bydef 2m-1$, where $m$ is the dimension of the finite dimensional projection $f^{(m)}$, one may choose  
$$
  Y_k={\tiny  \left( \hspace{-.1cm} \begin{array}{cc} 0 \\ 0 \end{array}\hspace{-.1cm} \right)}, \qquad\text{and}\qquad  Z_k = \hat{Z}_M \left(\frac{M^s}{\omega_k}\right),
$$
where $\hat{Z}_M(r,\dal) \gc {\tiny  \left( \hspace{-.1cm} \begin{array}{cc} 0 \\ 0 \end{array}\hspace{-.1cm} \right)}$ is independent of $k$. The choice $M=2m-1$ will be justified in Section~\ref{sec:Y}. This leads us to the following definition.
\begin{definition}\label{def:radpol}
Let $Y_k(\dal)={\tiny  \left( \hspace{-.1cm} \begin{array}{cc} 0 \\ 0 \end{array}\hspace{-.1cm} \right)}$ and $Z_k(r,\dal) = \hat{Z}_M(r,\dal) \left(\frac{M^s}{\omega_k}\right) $ for all $k \geq M$.
We define the $2M$ \emph{radii polynomials} $\{ p_0, \ldots,p_{M-1},p_M\}$ by
$$
 p_k(r,\dal) \bydef \left\{ \begin{array}{ll}
    Y_k(\dal)+Z_k(r,\dal)-\frac{r}{\omega_k}{\tiny  \left( \hspace{-.1cm} \begin{array}{cc} 1 \\ 1 \end{array}\hspace{-.1cm} \right)} , & k=0,\ldots,M-1; \\[1mm]  
     \hat{Z}_M(r,\dal) - \frac{r}{\omega_M} {\tiny  \left( \hspace{-.1cm} \begin{array}{cc} 1 \\ 1 \end{array}\hspace{-.1cm} \right)} & k = M.
  \end{array} \right. 
$$
\end{definition}

The following result was first considered in \cite{BLM}. 
\begin{lemma}
\label{lem:radpol}
If there
  exists an $r>0$ and $\dal \ge 0$ such that $p_k(r,\dal) <0$ for
  all $k=0,\ldots,M$, then there exist a $C^\infty $ function $\tx :
  [\alpha_0,\alpha_0+\dal] \rightarrow \Omega^s: \alpha \mapsto \tx(\alpha)$ such
  that $f(\tx(\alpha),\alpha)=0$ for all $\alpha \in
  [\alpha_0,\alpha_0+\dal]$. Furthermore, these are the only solutions
  of $f(x,\alpha)=0$ in the tube $\{ \alpha \in [\alpha_0,\alpha_0+\dal] , x-x_\alpha \in
  B(r) \}$.
\end{lemma}
\begin{proof}
By definition of the radii polynomials and because they satisfy $p_k(r,\dal) <0$ for
  all $k=0,\ldots,M$, and by the choice of $Y_k$ and $Z_k$ for $k \ge M$, we get that 
\[
\|Y + Z\|_s = \sup_{k=0,1,\dots} {\| Y_k(\dal)+Z_k(r,\dal) \|}_\infty \omega_k  < r .
\]
Since $p_k$ is increasing in $\dal \geq 0$ (see Remark~\ref{rem:pk_increasing}), existence and uniqueness of 
a solution $\tx(\alpha)$ for $\alpha \in [\alpha_0,\alpha_0+\dal]$
follows from Lemma~\ref{lem:ex_uni}. In particular, for every fixed $\alpha \in [\alpha_0,\alpha_0+\dal]$, $T(\cdot,\alpha):  B_{x_\alpha}(r) \rightarrow B_{x_\alpha}(r)$ is a contraction. Consider the change of variable $y=x- x_\alpha$. Then, the operator \[ \widetilde{T}:  [\alpha_0,\alpha_0+\dal] \times B(r) \rightarrow B(r):(\alpha,y) \mapsto \widetilde{T}(\alpha,y) \bydef T(y+x_\alpha,\alpha) \] 
is a uniform contraction on $B(r)$. Since $f  \in C^\infty \left( \Omega^s , \Omega^{s-1} \right)$, we have that $\widetilde{T} \in C^\infty \left([\alpha_0,\alpha_0+\dal] \times B(r) , B(r) \right)$. By the uniform contraction principle, we conclude that $\tx(\alpha)$ is a $C^\infty$ function of $\alpha$; see e.g. \cite{CH}.
\end{proof}

The remaining part of the section is taken almost verbatim from \cite{BLM}. 

In practice, we use an iterative procedure (with $\dal$ varying) to find the approximate
maximal $\dal^0$ (if it exists) for which there exists an $r>0$ such that the hypotheses of Lemma~\ref{lem:radpol} are satisfied. If this step is
successful, we let $\alpha_1 = \alpha_0+\dal^0$ and we obtained a continuum of zeros
$\C_0 = \left\{ \left(x^0(\alpha), \alpha \right) |~ f\left(x^0(\alpha), \alpha \right)=0,~~\alpha \in [\alpha_0,\alpha_1] \right\}$. We now want to repeat the argument with initial
parameter value $\alpha_1$. Hence, we put ourself in the context of a
continuation method, which involves a predictor and corrector step. Recalling the definition of the predictors based at $\alpha_0$ given by (\ref{eq:x_alpha}), the predictor at the parameter value
$\alpha_1=\alpha_0+\dal^0$ is given by $\hat{x}_1 \bydef \bx+\dal^0 \dx$. The
corrector step, based on a Newton-like iterative scheme on the projection
$f^{(m)}$, takes $\hat{x}_1$ as its input and produces, within a prescribed
tolerance, a zero $\bx_1$ at~$\alpha_1$. We can then compute a new tangent
vector $\dx_1$, built the new set of predictors $\bx_1+ \dal \dx_1$,
construct the bounds $Y,Z$ at the parameter value $\alpha_1$ and try to verify
the hypotheses of Lemma~\ref{lem:radpol} again. If we are successful in finding a new
$\dal^0$, we let $\alpha_2=\alpha_1+\dal^0$ and we get the existence of a
continuum of zeros $\C_1 = \left\{ \left( x^1(\alpha) ,\alpha \right) |~ f\left(
   x^1(\alpha), \alpha \right)=0,~~ \alpha \in [\alpha_1,\alpha_2] \right\}$. The question now is to determine whether or not
$\C_0$ and $\C_1$ connect at the parameter value $\alpha_1$ to form a continuum of zeros $\C_0 \cup \C_1$. At the parameter value
$\alpha_1$, we have two sets enclosing a unique zero
namely 
$$
  B_0 \bydef  \bx_0 + (\alpha_1-\alpha_0)\dx_0 +B(r_0), 
$$ 
and 
$$ 
  B_1 \bydef  \bx_1 + B(r_1).
$$
We want to prove that the solutions in $B_0$ and $B_1$ are the same. We return now to the radii polynomials
$p_k(r,\dal)$, $k=0,\dots,M$  constructed 
at basepoint $(x,\alpha)=(\bx_1,\alpha_1)$, and evaluate them at $\dal=0$:
$$
  \tilde{p}_k(r)=p_k(r,0).
$$
Since $\tilde{p}_k(r_1)<0$, we find a non empty interval $\cI \bydef
[r_1^-,r_1^+]$ containing $r_1$ such that $\tilde{p}_k(r)$ are all
strictly negative on $\cI$.
We now have two additional sets enclosing a unique zero at parameter value
$\alpha_1$, namely
$$
  B^\pm_1 \bydef  \bx_1 + B(r^\pm_1).  
$$
The proof of the following result can be found also in \cite{BLM}.
\begin{proposition} \label{prop:continuum} If $B_0 \subset
  B^+_1$ or $B^-_1 \subset B_0$, then
  $\C_0 \cup \C_1$ consists of a continuous branch of solutions of $f(x,\alpha)=0$, 
  and $\C_0 \cap \C_1 = \{(x^0(\alpha_1),\alpha_1\} =
  \{(x^1(\alpha_1),\alpha_1\} \in B_0 \cap B_1$. 
\end{proposition}
\begin{figure}
\centerline{\includegraphics[width=7cm]{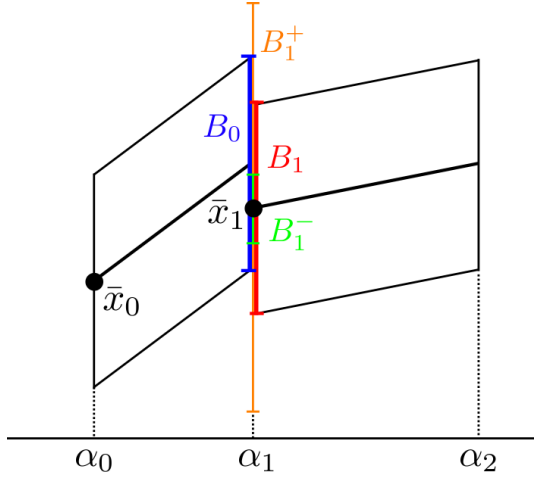}} 
\caption{$B_0 \cap B_1$ contains a unique zero of (\ref{eq:f}) and $\C_0 \cup \C_1$ consists of a continuum of zeros. This picture illustrates the hypotheses of Proposition~\ref{prop:continuum}.}
\label{fig:B_0B_1}
\end{figure}

We have now all the ingredients to prove Theorem~\ref{thm:SOPS}.

\section{The proof of Theorem~\ref{thm:SOPS} } \label{sec:proof}

The proof of Theorem~\ref{thm:SOPS} is {\em constructive} and it has two parts. The first one is a rigorous continuation in the parameter $\alpha \in [\pi/2+\eps,2.3]$ of a branch (denoted by $\cF^*_0$) of periodic solutions of (\ref{eq:wright}). This part of the proof is presented in Section~\ref{sec:1st_part}. The second part of the proof, presented in Section~\ref{sec:2nd_part}, verifies that $\cF^*_0 \subset \cF_0$. In other words, we prove that the global solution curve $\cF^*_0$, computed in the first part, belongs to the branch of SOPS that bifurcates from the trivial solution at $\alpha=\pi/2$.

Since we use validated continuation in the proof, we need to construct analytically the radii polynomials introduced in Definition~\ref{def:radpol}. Section~\ref{sec:Y} is dedicated to the computation of the bound $Y(\dal)$, defined component-wise by (\ref{eq:Y_k}), while Section~\ref{sec:Z} is dedicated to the computation of the bound $Z(r,\dal)$, defined component-wise by (\ref{eq:Z_k}).

%%%%%%%%
%%%  Y   %%%
%%%%%%%%

\subsection{The analytic bound \boldmath$Y(\dal)$\unboldmath} 
\label{sec:Y}

The goal of this section is to construct an analytic expression for the bound $Y=Y(\dal)$ given by (\ref{eq:Y_k}). Recall that this bound satisfies the following component-wise inequalities:
\[  \Bigl| [T(x_\alpha,\alpha)-x_\alpha]_k \Bigr| = \Bigl| \left[ -A f(x_\alpha,\alpha) \right]_k \Bigr| \lec Y_k(\dal). \]
As mentioned in Section~\ref{sec:validated_continuation}, for a fixed value of $\alpha_0$, we consider $\alpha \ge \alpha_0$ and we let $\dal=\alpha-\alpha_0 \ge 0$. As a side remark, note that once the analytic bound $Y_k=Y_k(\dal)=Y_k(\alpha-\alpha_0)$ is derived, we use a computer program using interval arithmetic to get explicit numerical upper bound for $Y_k$. By definition of $f_k$ given by (\ref{eq:f_k1}) and (\ref{eq:f_k2}), observe that $f_k(x_\alpha,\alpha)={\tiny  \left( \hspace{-.1cm} \begin{array}{cc} 0 \\ 0 \end{array}\hspace{-.1cm} \right)}$ for $k \ge 2m-1$. This is due to the fact that $[x_\alpha]_k = {\tiny  \left( \hspace{-.1cm} \begin{array}{cc} 0 \\ 0 \end{array}\hspace{-.1cm} \right)}$ for $k \ge m$. By definition of $A$ given by (\ref{eq:A}), one can choose $Y_k(\dal)={\tiny  \left( \hspace{-.1cm} \begin{array}{cc} 0 \\ 0 \end{array}\hspace{-.1cm} \right)}$, for $k \ge 2m-1$. This fact justifies the choice of $M \bydef 2m-1$ already introduced in Section~\ref{sec:validated_continuation}. Now that $Y_k$ is constructed for the cases $k \ge M$, we are left with the cases $0 \le k \le M-1$. Given $i \in \{1,2\}$ and $k \in \{0,\ldots,2m-2\}$, let us compute the analytic bound $Y_{k,i}(\dal)$. As mentioned already in Section~\ref{sec:validated_continuation}, we want to construct $Y_{k,i}(\dal)$ as a polynomial in $\dal$. Recalling (\ref{eq:Y_k}), we begin by splitting the expression $f(x_\alpha,\alpha)$ in two terms. The first term, very small because of the choices of $\bx$ from (\ref{eq:approx_zero}) and $\dx$ from (\ref{eq:tangent}), does not require any further analysis. The second term, not necessarily small, is expanded as an analytic polynomial using the software {\em Maple} and then bounded using further analysis.

Let us now expand $f(x_\alpha,\alpha)$ component-wise as powers of $\dal$ using the function
\[ 
h_{k,i}^Y(\alpha)\bydef f_{k,i} (x_\alpha,\alpha) =f_{k,i}(\bx+(\alpha-\alpha_0)\dx,\alpha).
\]
Recalling that $\dal=\alpha-\alpha_0 \ge 0$, Taylor's theorem implies the existence of $\alpha_{k,i}^* \in [\alpha_0,\alpha]$ such that 
\begin{eqnarray*} 
f_{k,i} (x_\alpha,\alpha) &=& h_{k,i}^Y(\alpha) = h_{k,i}^Y(\alpha_0) + \frac{d h^Y_{k,i}}{d \alpha} (\alpha_0) (\alpha-\alpha_0) +\frac{1}{2} \frac{d^2 h^Y_{k,i}}{d \alpha^2} (\alpha_{k,i}^*) (\alpha-\alpha_0)^2 \\
&=& f_{k,i}(\bx,\alpha_0) + \left[ Df_{k,i}(\bx,\alpha_0)\dx + \frac{\partial f_{k,i}}{\partial \alpha} (\bx,\alpha_0)\right] \dal +\frac{1}{2} \frac{d^2 h^Y_{k,i}}{d \alpha^2} (\alpha_{k,i}^*) \dal^2.
\end{eqnarray*}
Letting
\begin{equation} \label{eq:rd01hat2}
\begin{array}{lll}
\rd_{k,i}^{(0)} \bydef f_{k,i}(\bx,\alpha_0); \\
\rd_{k,i}^{(1)} \bydef Df_{k,i}(\bx,\alpha_0)\dx + \frac{\partial f_{k,i}}{\partial \alpha} (\bx,\alpha_0); \\
\hat{\rd}_{k,i}^{(2)}(\alpha_{k,i}^*) \bydef \frac{1}{2} \frac{d^2 h^Y_{k,i}}{d \alpha^2} (\alpha_{k,i}^*) 
\end{array}
\end{equation} 
we have, as wanted, the following polynomial expression for $f_{k,i}$, namely
\begin{equation} \label{eq:fki_poly}
f_{k,i} (x_\alpha,\alpha) = \rd_{k,i}^{(0)} + \rd_{k,i}^{(1)} \dal + \hat{\rd}_{k,i}^{(2)}(\alpha_{k,i}^*) \dal^2. 
\end{equation}
As mentioned above, the choice of the expansion (\ref{eq:fki_poly}) is made because the coefficients $\rd_{k,i}^{(0)}$ and $\rd_{k,i}^{(1)}$ from (\ref{eq:rd01hat2}) are small. Indeed, $\rd_{k,i}^{(0)}$ is small since $(\bx,\alpha_0)$ is a numerical approximation of (\ref{eq:approx_zero}) and $\rd_{k,i}^{(1)}$ is small because $\dx$ is a numerical approximation of (\ref{eq:tangent}). 
In practice, $\rd_{k,i}^{(0)}$ and $\rd_{k,i}^{(1)}$ are evaluated using interval arithmetic. Hence, one can compute an explicit numerical upper bound for each of them. However, we cannot evaluate the quadratic coefficient $\hat{\rd}_{k,i}^{(2)}(\alpha_{k,i}^*)$ of (\ref{eq:fki_poly}) in the same fashion, because it depends on the unknown $\alpha_{k,i}^* \in [\alpha_0,\alpha]=[\alpha_0,\alpha_0+\dal]$. The idea here is to define the quantity $\dal^{(k,i)} \bydef \alpha_{k,i}^* -\alpha_0 \in [0,\dal]$ and to expand $\hat{\rd}_{k,i}^{(2)}(\alpha_{k,i}^*)=\hat{\rd}_{k,i}^{(2)}(\alpha_0+\dal^{(k,i)})$ as powers of $\dal^{(k,i)}$. Once this expansion is done, the next step will be to use the fact that 
\begin{equation} \label{eq:dal_star_ineq}
0 \le \dal^{(k,i)} \le \dal,~ {\rm for~all~} i \in \{1,2\} ~ {\rm and}~ k \in \{0,\ldots,2m-2\}.
\end{equation}
We will come back to (\ref{eq:dal_star_ineq}) later. Using the mathematical software {\em Maple}, we compute analytic expressions $\rd_{k,i}^{(2)}$, $\rd_{k,i}^{(3)}$, $\rd_{k,i}^{(4)}$ and $\rd_{k,i}^{(5)}$ so that
\begin{equation} \label{eq:hatd2}
\hat{\rd}_{k,i}^{(2)}(\alpha_0+\dal^{(k,i)}) = \sum_{j = 2}^5 \rd_{k,i}^{(j)} \left( \dal^{(k,i)}\right)^{j-2}.
\end{equation}
The {\em Maple} program {\em D.mw} generating the $\rd_{k,i}^{(j)}$, $j=2,3,4,5$ can be found at \cite{code}. The first part of the program differentiate $h_{k,i}^Y(\alpha)\bydef f_{k,i} (x_\alpha,\alpha)$ twice with respect to $\alpha$ and then expands $\hat{\rd}_{k,i}^{(2)}(\alpha_0+\dal^{(k,i)})$ in powers of $\dal^{(k,i)}$. For more technical details about the expansion (\ref{eq:hatd2}), we refer again to \cite{code}. Combining (\ref{eq:fki_poly}) and (\ref{eq:hatd2}), one gets that
\[  f_{k,i} (x_\alpha,\alpha) = \sum_{j = 0}^1 \rd_{k,i}^{(j)} \dal^j + \sum_{j = 2}^5 \rd_{k,i}^{(j)} \left( \dal^{(k,i)}\right)^{j-2}\dal^2 . \]
As mentioned earlier, we now use property (\ref{eq:dal_star_ineq}) and get rid of the dependence of $f_{k,i} (x_\alpha,\alpha)$ in terms of $\dal^{(k,i)}$. In order to do so, let us define 
\[
\rd_{_F}^{(j)}=\left( (\rd_{0,1}^{(j)},\rd_{0,2}^{(j)}) ,(\rd_{1,1}^{(j)},\rd_{1,2}^{(j)}),\dots,(\rd_{2m-2,1}^{(j)},\rd_{2m-2,2}^{(j)}) \right)^T, ~j=0,\dots,5.
\]
For $j=2,3,4,5$, let $\trd_{k,i}^{(j)} \bydef \rd_{k,i}^{(j)} \left( \dal^{(k,i)}\right)^{j-2}$ and 
\[
\trd_{_F}^{(j)}=\left( (\trd_{0,1}^{(j)},\trd_{0,2}^{(j)}) ,(\trd_{1,1}^{(j)},\trd_{1,2}^{(j)}),\dots,(\trd_{2m-2,1}^{(j)},\trd_{2m-2,2}^{(j)}) \right)^T, ~j=2,3,4,5.
\]
For the cases $k=0,\dots,2m-2$, we combine (\ref{eq:dal_star_ineq}) and triangle inequality to obtain that
\begin{eqnarray*}
\left| [T(x_\alpha,\alpha)-x_\alpha]_{_F} \right| &=&
\left| -J_{_F} f_{_F}(x_\alpha,\alpha) \right| \\
&=& \left| \sum_{j = 0}^1 J_{_F}  \rd_{_F}^{(j)} \dal^j + \sum_{j = 2}^5 J_{_F}  \trd_{_F}^{(j)} \dal^2 \right| \\
& \lec & \sum_{j = 0}^1 \left| J_{_F}  \rd_{_F}^{(j)} \right| \dal^j + \sum_{j = 2}^5 \left| J_{_F} \right|  \left| \rd_{_F}^{(j)} \right| \dal^j .
\end{eqnarray*}
As we mentioned before, the first part of the {\em Maple} program {\em D.mw} symbolically computes $\rd_{_F}^{(j)}$, for $j=2,3,4,5$. The second part of {\em D.mw} helps obtaining the analytic upper bounds $D_k^{(j)}$ ($j=2,3,4,5$) such that for $i=1,2$, $| \rd_{k,i}^{(j)} | \le D_k^{(j)}$. The bounds $D_k^{(j)}$ are presented in Table~\ref{t:bounds_Dk}. It is important to note that all sums presented in Table~\ref{t:bounds_Dk} are finite sums. Hence, we can use a computer to compute them rigorously using interval arithmetic. Note also that $D_{0,1}^{(j)}=0$ for all $j=2,3,4,5$.
%%%% Bounds D %%%%
\begin{table}
\newcommand{\cocc}[3]{$D_{#1}^{(#2)}$ & $#3$ \\[1mm] \hline}
\renewcommand{\arraystretch}{1.7}
\centerline{\tiny \begin{tabular}{|r@{\hspace{2mm}}|@{\hspace{2mm}}l|}
\hline
\multicolumn{2}{|c|}{$k = 0,\dots,2m-2$} \\ \hline
\cocc{k}{2}{ 
\begin{array}{llll}
k |\dL| \left( |\da_k|+ |\db_k| \right) +
\left| k\dL\ba_k + \frac{1}{2}\alpha_0 k^2 \dL^2 \bb_k + \alpha_0 k \dL \da_k - \db_k \right| +
\left| \da_k -\frac{1}{2}\alpha_0 k^2 \dL^2 \ba_k + k \dL \bb_k + \alpha_0 k \dL \db_k \right| 
\\
+{\displaystyle \sum_{k_1+k_2=k}} 
\left| -k_1 \dL \left( \ba_{k_1}\ba_{k_2}-\bb_{k_1}\bb_{k_2} \right) + \left( \ba_{k_1}\db_{k_2}+\da_{k_1}\bb_{k_2}+\bb_{k_1}\da_{k_2}+\db_{k_1}\ba_{k_2} \right) - \frac{1}{2}\alpha_0 k_1^2 \dL^2 \left( \ba_{k_1}\bb_{k_2}+\bb_{k_1}\ba_{k_2} \right) \right. \\
\left. \hspace{1.5cm}  -\alpha_0 k_1\dL \left( \ba_{k_1}\da_{k_2}+\da_{k_1}\ba_{k_2}-\bb_{k_1}\db_{k_2}-\db_{k_1}\bb_{k_2} \right) + \alpha_0  \left( \da_{k_1}\db_{k_2} + \db_{k_1}\da_{k_2} \right) \right| \\
+ {\displaystyle \sum_{k_1+k_2=k}}  \left| k_1\dL \left( \ba_{k_1}\bb_{k_2}+\bb_{k_1}\ba_{k_2} \right) + \left( \ba_{k_1}\da_{k_2}+\da_{k_1}\ba_{k_2}-\bb_{k_1}\db_{k_2}-\db_{k_1}\bb_{k_2} \right) 
- \frac{1}{2}\alpha_0 k_1^2 \dL^2 \left( \ba_{k_1}\ba_{k_2}-\bb_{k_1}\bb_{k_2} \right) \right.
\\
\left. \hspace{1.5cm} + \alpha_0 \left( \da_{k_1}\da_{k_2}- \db_{k_1}\db_{k_2} \right) + \alpha_0 k_1\dL \left( \ba_{k_1}\db_{k_2}+\da_{k_1}\bb_{k_2}+\bb_{k_1}\da_{k_2}+\db_{k_1}\ba_{k_2} \right) \right|
\end{array}
}
\cocc{k}{3}{ 
\begin{array}{lllll}
\left| 2 k \dL \da_k + \frac{1}{2}\alpha_0 k^2 \dL^2 \db_k + \frac{1}{2} k^2 \dL^2 \bb_k \right|
+ \left|  2 k\dL \db_k -\frac{1}{2}\alpha_0 k^2 \dL^2 \da_k - \frac{1}{2} k^2 \dL^2 \ba_k \right|
\\
+{\displaystyle \sum_{k_1+k_2=k}} 
\left| 2 k_1\dL \left( \ba_{k_1}\db_{k_2}+\da_{k_1}\bb_{k_2}+\bb_{k_1}\da_{k_2}+\db_{k_1}\ba_{k_2} \right) + 3 \left( \da_{k_1}\da_{k_2} - \db_{k_1}\db_{k_2} \right) 
+ 2 \alpha_0 k_1\dL \left( \da_{k_1}\db_{k_2} + \db_{k_1}\da_{k_2} \right) \right. 
\\
\left. \hspace{1.5cm} - \frac{1}{2}\alpha_0 k_1^2 \dL^2 \left( \ba_{k_1}\da_{k_2}+\da_{k_1}\ba_{k_2}-\bb_{k_1}\db_{k_2}-\db_{k_1}\bb_{k_2} \right) -\frac{1}{2} k_1^2\dL^2 \left( \ba_{k_1}\ba_{k_2}-\bb_{k_1}\bb_{k_2} \right) \right|
\\
+ {\displaystyle \sum_{k_1+k_2=k}} 
\left| - 2 k_1\dL \left( \ba_{k_1}\da_{k_2}+\da_{k_1}\ba_{k_2}-\bb_{k_1}\db_{k_2}-\db_{k_1}\bb_{k_2} \right) + 3 \left( \da_{k_1}\db_{k_2}+\db_{k_1}\da_{k_2} \right) 
-2\alpha_0 k_1\dL \left( \da_{k_1}\da_{k_2}- \db_{k_1}\db_{k_2} \right) \right.
\\
\left. \hspace{1.5cm} -\frac{1}{2}\alpha_0 k_1^2\dL^2 \left( \ba_{k_1}\db_{k_2}+\da_{k_1}\bb_{k_2}+\bb_{k_1}\da_{k_2}+\db_{k_1}\ba_{k_2} \right) 
- \frac{1}{2} {k_1}^{2}{\dL}^{2} \left( \ba_{k_1}\bb_{k_2}+\bb_{k_1}\ba_{k_2} \right) \right|
\end{array}
}
\cocc{k}{4}{ 
\begin{array}{llll}
\frac{1}{2} k^{2}{\dL}^{2} \left( |\db_k|+ |\da_k| \right) 
\\
+ {\displaystyle \sum _{k_1+k_2=k}} 
\left| 3 k_1\dL \left( \da_{k_1}\da_{k_2}-\db_{k_1}\db_{k_2} \right) 
+ \frac{1}{2}\alpha_0 k_1^2 \dL^2 \left( \da_{k_1}\db_{k_2}+\db_{k_1}\da_{k_2} \right) + \frac{1}{2} k_1^2 \dL^2 \left( \ba_{k_1}\db_{k_2}+\da_{k_1}\bb_{k_2}+\bb_{k_1}\da_{k_2}+\db_{k_1}\ba_{k_2} \right) \right|
\\
+ {\displaystyle \sum _{k_1+k_2=k}} 
\left| 3 k_1\dL \left( \da_{k_1}\db_{k_2}+\db_{k_1}\da_{k_2} \right) 
- \frac{1}{2}\alpha_0  k_1^2 \dL^2 \left( \da_{k_1}\da_{k_2}-\db_{k_1}\db_{k_2} \right) - \frac{1}{2} k_1^2 \dL^2 \left( \ba_{k_1}\da_{k_2}+\da_{k_1}\ba_{k_2}-\bb_{k_1}\db_{k_2}-\db_{k_1}\bb_{k_2} \right) \right|
\end{array}
}
\cocc{k}{5}{ 
{\displaystyle \sum _{k_1+k_2=k}} \frac{1}{2} {k_1}^{2}{\dL}^{2} \left[ |  \da_{k_1}\db_{k_2}+\db_{k_1}\da_{k_2} | + | \da_{k_1}\da_{k_2}-\db_{k_1}\db_{k_2} | \right]
} 
\end{tabular}}
\caption{The bounds $D_{k}^{(j)}$.}
\label{t:bounds_Dk}
\end{table}
Letting
\[
Y_{_F}^{(j)} \bydef \left\{ 
\begin{array}{ll} 
| J_{_F}  \rd_{_F}^{(j)} | ~~, ~~j=0,1 \\
 \left| J_{_F} \right|  D_{_F}^{(j)}~,~~j=2,3,4,5
 \end{array} \right.
\]
we can finally set
\begin{equation} \label{eq:Y_F}
Y_{_F}(\dal) \bydef \sum_{j = 0}^5 Y_{_F}^{(j)} \dal^j.
\end{equation}
%

%%%%%%%%%%
%%%%  Z   %%%%
%%%%%%%%%%

\subsection{The analytic bound \boldmath$Z(r,\dal)$\unboldmath} 
\label{sec:Z}

In this section, we construct analytically the bound $Z=Z(r,\dal)$. Recall from (\ref{eq:Z_k}) that this bound satisfies the component-wise inequalities
\[
\sup_{b,c \in B(r)} \Bigl| [D_xT(x_\alpha+b,\alpha)c]_k \Bigr| = \sup_{u,v \in B(1)} \Bigl| [D_xT(x_\alpha+ru,\alpha)rv]_k\Bigr|\lec Z_k(r,\dal) .
\]
As mentioned previously in Section~\ref{sec:validated_continuation}, we are going to construct each component $Z_{k,i}(r,\dal)$ ($i=1,2$, $k \ge 0$) of $Z(r,\dal)$ as a polynomial in the variables $r$ and $\dal$. In spirit, the construction of the polynomial expansion of $Z(r,\dal)$ is similar to the construction of the polynomial expansion of $Y(\dal)$ of Section~\ref{sec:Y}. We begin by splitting the expression $D_xT(x_\alpha+ru,\alpha)rv$ in two terms. The first term is small and does not require any further analysis. The second term, on the other hand, requires more analysis. It is expanded as an analytic polynomial using the software {\em Maple} and then bounded using analytic estimates. Let us now be more explicit.

Introducing an almost inverse of the operator $A$ defined in~(\ref{eq:A})
$$
\dagA \bydef \left[\begin{array}{cccccc}
D_{x}f^{(2m-1)}(\bx,\alpha_0) &\,& \zft & \zft & \zft & \cdots \\[2mm]
\zf &\,& \Lambda_{2m-1} & 0 &0 &\cdots \\
\zf &\,& 0 & \Lambda_{2m} &0 & \cdots \\
\zf &\,& 0 & 0& \Lambda_{2m+1} &  \\
\vdots         &\,& \vdots & \vdots & & \ddots  
        \end{array}\right] \, 
$$
we can split $Df(x_\alpha+ru,\alpha)rv$ into two pieces
$$
D_xf(x_\alpha+ru,\alpha)rv = \dagA rv + \bigl[ D_xf(x_\alpha+ru,\alpha)rv - \dagA rv \bigr] .
$$
Hence, we get
\begin{equation}\label{e:splitT}
D_xT(x_\alpha+ru,\alpha)rv = \left( [I - A \dagA] v \right) r - A \, [ D_xf(x_\alpha+ru,\alpha)- \dagA \bigr] v r.
\end{equation}
Note that the infinite dimensional vector $[I - A \dagA] v$ has only finitely many nonzero entries and its finite non trivial part, given by $\left[ I_{_F}-J_{_F}  D_xf^{(2m-1)}(\bx,\alpha_0) \right] v_{_F} \in \mathbb{R}^{2(2m-1)}$, has a small magnitude. This is due to the fact that $J_{_F}$ is a numerical approximation of the inverse of $D_xf^{(2m-1)}(\bx_{_F},\alpha_0)$. In order to bound the second term of (\ref{e:splitT}), further analysis is required. The idea is the following. First, expand each component of the term $[D_xf(x_\alpha+ru,\alpha)- \dagA]vr$ as a finite polynomial of the form
\[
\left( [ D_xf(x_\alpha+ru,\alpha)- \dagA \bigr] vr \right)_{k.i} = \sum_{l_1,l_2} c_{k,i}^{(l_1,l_2)} r^{l_1} \dal^{l_2}.
\]
Second, compute analytic upper bounds $C_k^{(l_1,l_2)}$ so that $|c_{k,i}^{(l_1,l_2)}| \le C_k^{(l_1,l_2)}$ (uniform with respect to $i=1,2$). Finally, use the $C_k^{(l_1,l_2)}$ to define the polynomial bound $Z(r,\dal)$. 

The computation of the $c_{k,i}^{(l_1,l_2)}$ is done analytically using the {\em Maple} program {\em C.mw} which can be found at \cite{code}. The first part of the program computes an analytic representation of $\left( D_xf(x_\alpha+ru,\alpha) vr \right)_{k,i}$. Then, ignoring the fact that  the $\sin(\cdot)$ and the $\cos(\cdot)$ terms (coming from differentiating (\ref{eq:fk1}) and (\ref{eq:fk2})) depend also on $r$ and $\dal$, it computes analytically, for all $k \ge 0$ and $i \in \{1,2\}$ the polynomial expansion
\begin{equation} \label{eq:coeff_expansion}
\left( [ D_xf(x_\alpha+ru,\alpha)- \dagA \bigr] vr \right)_{k.i} = \sum_{l_1=1}^3 \sum_{l_2=0}^{4-l_1} c_{k,i}^{(l_1,l_2)} r^{l_1} \dal^{l_2}.
\end{equation}
Note that the coefficients $c_{k,i}^{(l_1,l_2)}$ of (\ref{eq:coeff_expansion}) are still depending on the $\sin(\cdot)$ and the $\cos(\cdot)$, which themselves depend on $r$ and $\dal$. The last part of {\em C.mw} is dedicated to the computation of the bounds $C_k^{(l_1,l_2)} \ge 0$ such that $ \left| c_{k,i}^{(l_1,l_2)} \right| \le C_{k}^{(l_1,l_2)}$, for $i=1,2$. This part of the program uses several times the triangle inequality and the fact that $|\sin|,|\cos| \le 1$. The bounds $C_{k}^{(l_1,l_2)}$ are presented in Table~\ref{t:boundsC}. 
%
%%%% Bounds C %%%%
\begin{table}
\newcommand{\cocc}[4]{$C_{#1}^{(#2,#3)}$ & $#4$ \\[1mm] \hline}
\renewcommand{\arraystretch}{1.7}
\centerline{\tiny \begin{tabular}{|r@{\hspace{2mm}}|@{\hspace{2mm}}l|}
\hline
%%%%%%%%%%%%%%%%%%%%%%%%%%%%%%%%%%%%%%%%%%%%%%%%%%%%%%%%%%
\multicolumn{2}{|c|}{$k=0$, $i=1$} \\ \hline
%%%%%%%%%%%%%%%%%%%%%%%%%%%%%%%%%%%%%%%%%%%%%%%%%%%%%%%%%%
\cocc{0}{1}{0}{
2 \displaystyle{\sum_{k=2m-1}^\infty} \frac{1}{\omega_k}
}
%%%%%%%%%%%%%%%%%%%%%%%%%%%%%%%%%%%%%%%%%%%%%%%%%%%%%%%%%%
\multicolumn{2}{|c|}{$k\in \{0,\dots,2m-2\}$.} \\ \hline
%%%%%%%%%%%%%%%%%%%%%%%%%%%%%%%%%%%%%%%%%%%%%%%%%%%%%%%%%%
\cocc{k}{1}{0}{ 
4 {  \alpha_0} \displaystyle{\sum _{{  k_1}=2 m-1}^{k+m-1}} {\frac { \left| {  \ba}_{{k-{  k_1}}} \right| + \left| {  \bb}_{{k-{  k_1}}} \right| }{\omega_{{{  k_1}}}}}
}
\cocc{k}{1}{1}{ 
\begin{array}{lllll}
k \left| {  \alpha_0} {  \db}_{{k}}+{  \bb}_{{k}} \right| +k \left| {  \alpha_0} {  \da}_{{k}}+{  \ba}_{{k}} \right| +\frac{2}{\omega_k}+{  \alpha_0} {k}^{2} \left| {  \dL} \right| \left(  \left| {  \ba}_{{k}} \right| + \left| {  \bb}_{{k}} \right|  \right) +2 {\frac {{  \alpha_0} k \left| {  \dL} \right| }{\omega_{{k}}}}+k \left( \left| \da_k \right|+  \left| {  \db}_{{k}} \right| \right) +{\frac {k \left| {  \dL} \right| }{\omega_{{k}}}} 
\\
+\displaystyle{\sum _{k_1=-m+1+k}^{m-1}} \left| {  k_1} \right|  \left| {  \ba}_{{{  k_1}}}{  \bb}_{{{  k-k_1}}}+{  \bb}_{{{  k_1}}}{  \ba}_{{{  k-k_1}}}+{  \alpha_0}  \left( {  \ba}_{{{  k_1}}}{  \db}_{{{  k-k_1}}}+{  \da}_{{{  k_1}}}{  \bb}_{{{  k-k_1}}}+{  \bb}_{{{  k_1}}}{  \da}_{{{  k-k_1}}}+{  \db}_{{{  k_1}}}{  \ba}_{{{  k-k_1}}} \right)  \right| 
\\
 +\displaystyle{\sum _{k_1=-m+1+k}^{m-1}} \left| {  k_1} \right|  \left| -{  \ba}_{{{  k_1}}}{  \ba}_{{{  k-k_1}}}+{  \bb}_{{{  k_1}}}{  \bb}_{{{  k-k_1}}}-{  \alpha_0}  \left( {  \ba}_{{{  k_1}}}{  \da}_{{{  k-k_1}}}+{  \da}_{{{  k_1}}}{  \ba}_{{{  k-k_1}}}-{  \bb}_{{{  k_1}}}{  \db}_{{{  k-k_1}}}-{  \db}_{{{  k_1}}}{  \bb}_{{{  k-k_1}}} \right)  \right| 
 \\
+\displaystyle{\sum _{k_1=-m+1}^{m-1}} 4 {\frac { \left| {  \ba}_{{{  k_1}}}+{  \alpha_0} {  \da}_{{{  k_1}}} \right| + \left| {  \bb}_{{{  k_1}}}+{  \alpha_0} {  \db}_{{{  k_1}}} \right| }{\omega_{{{  k-k_1}}}}} 
+\displaystyle{\sum _{k_1=-m+1+k}^{m-1}}{  \alpha_0} {{  k_1}}^{2} \left| {  \dL} \right|  \left(  \left| -{  \ba}_{{{  k_1}}}{  \ba}_{{{  k-k_1}}}+{  \bb}_{{{  k_1}}}{  \bb}_{{{  k-k_1}}} \right| + \left| {  \ba}_{{{  k_1}}}{  \bb}_{{{  k-k_1}}}+{  \bb}_{{{  k_1}}}{  \ba}_{{{  k-k_1}}} \right|  \right) 
\\
+\displaystyle{\sum _{{  k_1}=-m+1+k}^{m-1}} {\frac {{ 2 \alpha_0} 
 \left| {  \dL} \right|  \left(  \left| {  k_1} \right| + \left| k-{  k_1} \right|  \right)  \left(  \left| {  \ba}_{{{  k_1}}} \right| + \left| {  \bb}_{{{  k_1}}} \right|  \right) }{\omega_{{k-{  k_1}}}}}
\end{array}
}
\cocc{k}{1}{2}{ 
\begin{array}{ll}
k \left(  \left| {  \db}_{{k}} \right| + \left| {  \da}_{{k}} \right|  \right) 
+\displaystyle{\sum _{k_1=-m+1+k}^{m-1}} \left| {  k_1} \right|  \left| {  \ba}_{{{  k_1}}}{  \db}_{{{  k-k_1}}}+{  \da}_{{{  k_1}}}{  \bb}_{{{  k-k_1}}}+{  \bb}_{{{  k_1}}}{  \da}_{{{  k-k_1}}}+{  \db}_{{{  k_1}}}{  \ba}_{{{  k-k_1}}}+{  \alpha_0}  \left( {  \da}_{{{  k_1}}}{  \db}_{{{  k-k_1}}}+{  \db}_{{{  k_1}}}{  \da}_{{{  k-k_1}}} \right)  \right| 
\\
+\displaystyle{\sum _{k_1=-m+1+k}^{m-1}} \left| {  k_1} \right|  \left| -{  \ba}_{{{  k_1}}}{  \da}_{{{  k-k_1}}}-{  \da}_{{{  k_1}}}{  \ba}_{{{  k-k_1}}}+{  \bb}_{{{  k_1}}}{  \db}_{{{  k-k_1}}}+{  \db}_{{{  k_1}}}{  \bb}_{{{  k-k_1}}}-{  \alpha_0}  \left( {  \da}_{{{  k_1}}}{  \da}_{{{  k-k_1}}}-{  \db}_{{{  k_1}}}{  \db}_{{{  k-k_1}}} \right)  \right| 
+\displaystyle{\sum _{k_1=-m+1}^{m-1} } 4 {\frac { \left| {  \da}_{{{  k_1}}} \right| + \left| {  \db}_{{{  k_1}}} \right| }{\omega_{{{ k- k_1}}}} }
\end{array}
}
\cocc{k}{1}{3}{ 
\displaystyle{\sum _{k_1=-m+1+k}^{m-1}} \left| {  k_1} \right|  \left(  \left| -{  \da}_{{{  k_1}}}{  \da}_{{{  k-k_1}}}+{  \db}_{{{  k_1}}}{  \db}_{{{  k-k_1}}} \right| + \left| {  \da}_{{{  k_1}}}{  \db}_{{{  k-k_1}}}+{  \db}_{{{  k_1}}}{  \da}_{{{  k-k_1}}} \right|  \right) 
}
\cocc{k}{2}{0}{
\begin{array}{ll}
{\frac {{ 4 \alpha_0} k}{\omega_{{k}}}}+{  \alpha_0} {k}^{2} \left(  \left| {  \ba}_{{k}} \right| + \left| {  \bb}_{{k}} \right|  \right) +2 {\frac {k}{\omega_{{k}}}}
+ \displaystyle{\sum _{k_1+k_2=k}} {\frac {{ 8 \alpha_0}}{\omega_{{{  k_1}}}\omega_{{{  k_2}}}}}
+\displaystyle{\sum _{{  k_1}=-m+1}^{m-1}} {\frac {{ 2 \alpha_0}  \left(  \left| {  k_1} \right| + \left| k-{  k_1} \right|  \right)  \left(  \left| {  \ba}_{{{  k_1}}} \right| + \left| {  \bb}_{{{  k_1}}} \right|  \right) }{\omega_{{k-{  k_1}}}}}
\\
+\displaystyle{\sum _{{  k_1}=-m+1+k}^{m-1}} {\frac {{ 2 \alpha_0}  \left(  \left| {  k_1} \right| + \left| k-{  k_1} \right|  \right)  \left(  \left| {  \ba}_{{{  k_1}}} \right| + \left| {  \bb}_{{{  k_1}}} \right|  \right) }{\omega_{{k-{  k_1}}}}}
+\displaystyle{\sum _{k_1=-m+1+k}^{m-1}}{  \alpha_0} {{  k_1}}^{2} \left(  \left| -{  \ba}_{{{  k_1}}}{  \ba}_{{{  k-k_1}}}+{  \bb}_{{{  k_1}}}{  \bb}_{{{  k-k_1}}} \right| + \left| {  \ba}_{{{  k_1}}}{  \bb}_{{{  k-k_1}}}+{  \bb}_{{{  k_1}}}{  \ba}_{{{  k-k_1}}} \right|  \right) 
\end{array}
} 
\cocc{k}{2}{1}{
2 {\frac {k}{\omega_{{k}}}}+ \displaystyle{\sum _{k_1+k_2=k}} {\frac {8}{\omega_{{{  k_1}}}\omega_{{{  k_2}}}}}
+ \displaystyle{\sum _{k_1=-m+1}^{m-1}} 2  \left( \left| {  k_1} \right| + \left| { k- k_1} \right| \right) \left( {\frac { \left| {  \ba}_{{{  k_1}}}+{  \alpha_0} {  \da}_{{{  k_1}}} \right| + \left| {  \bb}_{{{  k_1}}}+{  \alpha_0} {  \db}_{{{  k_1}}} \right| }{\omega_{{{  k-k_1}}}}} \right) 
}
\cocc{k}{2}{2}{ 
\displaystyle{\sum _{k_1=m-1}^{m-1}} 2 \left( \left| {  k_1} \right| + \left| {  k-k_1} \right| \right) \left( {\frac { \left| {  \da}_{{{  k_1}}} \right| + \left| {  \db}_{{{  k_1}}} \right| }{\omega_{{{  k-k_1}}}}} \right) 
}
\cocc{k}{3}{0}{ 
\displaystyle{\sum _{k_1+k_2=k}} {\frac {{ 4 \alpha_0}  \left| {  k_1} \right| }{\omega_{{{  k_1}}}\omega_{{{  k_2}}}}}
}
\cocc{k}{3}{1}{ 
\displaystyle{\sum _{k_1+k_2=k} {\frac { 4 \left| {  k_1} \right| }{\omega_{{{  k_1}}}\omega_{{{  k_2}}}}}}
}
\end{tabular}}
\caption{The bounds $C_{k,i}^{(l_1,l_2)}$ for $k=0,\dots,M-1$.}
\label{t:boundsC}
\end{table}
%%%%%%%%%%%%%%
%
Note that the cases $C^{(1,0)}_{0,1}$ and $C^{(1,0)}_{0,2}$ are treated differently. Indeed, the upper bound $|c^{(1,0)}_{0,1}| \le C^{(1,0)}_{0,1}$ is given in the first line of Table~\ref{t:boundsC} and for the upper bound $|c^{(1,0)}_{0,2}| \le C^{(1,0)}_{0,2}$, we use the bound $C^{(1,0)}_k$ (letting $k=0$, this bounds is actually $0$) on the second line of Table~\ref{t:boundsC}. Now that we have the bounds $C_k^{(l_1,l_2)}$, we are ready to compute the bounds $Z_k(r,\dal)$. 

\subsubsection{The analytic bounds \boldmath$Z_k(r,\dal)$, $k \in \{0,\dots, M-1\}$ \unboldmath}

As mentioned earlier, the {\em Maple} program {\em C.mw} generates the coefficients $C_k^{(l_1,l_2)}$.  Defining $C_{_F}^{(l_1,l_2)}={\tiny  \left( \hspace{-.2cm} \begin{array}{cc} C_k^{(l_1,l_2)} \\ C_k^{(l_1,l_2)} \end{array}\hspace{-.2cm} \right)}_{k=0,\dots,M-1}$, we get that
\begin{eqnarray*}
&& \hspace{-1.8cm} \left| \left[ D_xT(x_\alpha+ru,\alpha)rv \right]_{_F} \right|  \\
&=&  \left| \left[ I_{_F} - J_{_F}  D_xf^{(2m-1)}(\bx,\alpha_0) \right] v_{_F} r - J_{_F}  \left[ D_xf(x_\alpha+ru,\alpha)rv - \dagA rv\right]_{_F} \right| \\
&\lec& \left| \left[I_{_F} - J_{_F}  D_xf^{(2m-1)}(\bx,\alpha_0)\right] v_{_F} \right| r  + \sum_{l_1=1}^3 \sum_{l_2=0}^{4-l_1} |J_{_F}|  |c_{_F}^{(l_1,l_2)}| r^{l_1} \dal^{l_2} \\
&\lec& \left| \left[I_{_F} - J_{_F}  D_xf^{(2m-1)}(\bx,\alpha_0)\right] v_{_F} \right| r  + \sum_{l_1=1}^3 \sum_{l_2=0}^{4-l_1} |J_{_F}|  C_{_F}^{(l_1,l_2)} r^{l_1} \dal^{l_2}.
\end{eqnarray*}

Before proceeding further, it is important to remark that the coefficients $C_0^{(1,0)}$, $C_k^{(2,0)}$,$C_k^{(2,1)}$, $C_k^{(3,0)}$ and $C_k^{(3,1)}$ of Table~\ref{t:boundsC} involve infinite sums. This means that we have to use analytic estimates to bound these sums. The case of $C_0^{(1,0)}$ is trivial. For instance, consider the estimate 
\begin{equation} \label{eq:upper_boundS3}
\sum_{k=M}^\infty \frac{1}{\omega_k} \le \frac{1}{(s-1)(M-1)^{s-1}}
\end{equation}
The infinite sums involved in $C_k^{(2,0)}$,$C_k^{(2,1)}$, $C_k^{(3,0)}$ and $C_k^{(3,1)}$ can be bounded using the following result.

\begin{lemma} \label{lem:estimates_sum_finite_case}
Let $k \in \{0,\dots,M-1\}$, recall the definition of the weights $\omega_k$ in (\ref{e:weights}) and define
\begin{equation} \label{eq:phi_k}
\phi_k = \sum_{k_1=1}^{k-1} \frac{1}{k_1^s(k-k_1)^s}.
\end{equation}
Then
\begin{equation} \label{eq:upper_boundS1}
\sum _{k_1+k_2=k} \frac {1}{\omega_{k_1} \omega_{k_2}} \le \phi_k + \frac{1}{\omega_k} \left( 4 + \frac{2}{s-1} \right) 
\end{equation}
and
\begin{equation} \label{eq:upper_boundS2}
\sum _{k_1+k_2=k} \frac {|k_1|}{\omega_{k_1} \omega_{k_2}} \le \frac{1}{(k+1)^s} \left( 1 + \frac{1}{s-2} \right) +  \frac{k}{2} \phi_k +  \frac{k}{\omega_k} + \frac{1}{(k+1)^{s-1}} \left( 1 + \frac{1}{s-1} \right).
\end{equation}
\end{lemma}

\begin{proof}
First,
\begin{eqnarray*}
\sum _{k_1+k_2=k} \frac {1}{\omega_{k_1} \omega_{k_2}} &=& \sum_{k_1=-\infty}^{-1} \frac{1}{\omega_{k_1} \omega_{k-k_1}} + \frac{1}{\omega_k} + \sum_{k_1=1}^{k-1} \frac{1}{k_1^s(k-k_1)^s} +  \frac{1}{\omega_k} +  \sum_{k_1=k+1}^{\infty} \frac{1}{\omega_{k_1} \omega_{k-k_1}} 
\\
&=&  \phi_k + \frac{2}{\omega_k} + 2 \sum_{k_1=1}^\infty \frac{1}{k_1^s(k+k_1)^s} 
\\
&\le & \phi_k + \frac{1}{\omega_k} \left( 4 + \frac{2}{s-1} \right).
\end{eqnarray*} 
Second,
\begin{eqnarray*}
\sum _{k_1+k_2=k} \frac {|k_1|}{\omega_{k_1} \omega_{k_2}} &=&  \sum_{k_1=-\infty}^{-1} \frac{|k_1|}{\omega_{k_1} \omega_{k-k_1}} + \sum_{k_1=1}^{k-1} \frac{1}{k_1^{s-1}(k-k_1)^s} +  \frac{k}{\omega_k} +  \sum_{k_1=k+1}^{\infty} \frac{|k_1|}{\omega_{k_1} \omega_{k-k_1}} \\
&=&
\sum_{k_1=1}^{\infty} \frac{1}{k_1^{s-1} (k+k_1)^s} +  \frac{k}{2} \phi_k + \frac{k}{\omega_k} + \sum_{k_1=1}^{\infty} \frac{1}{(k+k_1)^{s-1} k_1^s} \\
&\le& \frac{1}{(k+1)^s} \left( 1 + \frac{1}{s-2} \right) +  \frac{k}{2} \phi_k +  \frac{k}{\omega_k} + \frac{1}{(k+1)^{s-1}} \left( 1 + \frac{1}{s-1} \right).
\end{eqnarray*} 
\end{proof}
Hence, replacing the infinite sums of sums of Table~\ref{t:boundsC} using the upper bounds (\ref{eq:upper_boundS3}), (\ref{eq:upper_boundS1}) and (\ref{eq:upper_boundS2}), we get new upper bounds ${\bf C}_{_F}^{(l_1,l_2)}$. For $k\in \{0,\dots,M-1\}$, we then define the $Z_k(r,\dal) \in \mathbb{R}^2$ to be the 2 dimensional $k^{th}-$ component of 
\begin{equation} \label{eq:Z_F}
Z_{_F}(r,\dal) \bydef \left| \left[I_{_F} - J_{_F}  Df^{(2m-1)}(\bx,\alpha_0)\right] v_{_F} \right| r  + \sum_{l_1=1}^3 \sum_{l_2=0}^{4-l_1} |J_{_F}|  {\bf C}_{_F}^{(l_1,l_2)} r^{l_1} \dal^{l_2}.
\end{equation}

\subsubsection{The analytic bound \boldmath$\hat{Z}_M(r,\dal)$ \unboldmath}

Consider $k \ge M=2m-1$. The goal of this section is to compute upper bounds $\hat{C}^{(l_1,l_2)}>0$ such that for every $k \ge M$ and  $i \in \{1,2\}$, 
\begin{equation} \label{eq:hatC}
\left| c_{k,i}^{(l_1,l_2)} \right| \le \frac{1}{k^{s-1}} \hat{C}^{(l_1,l_2)} 
\end{equation}
where $\hat{C}^{(l_1,l_2)}$ is independent of $k$ and $i$. We computed the $\hat{C}^{(l_1,l_2)}$ using the {\em Maple} program {\em hatC.mw} which can be found at \cite{code} and by using the following result. 

\begin{lemma} \label{lem:estimates_sum_tail_case}
Defining
\[ 
 \gamma \bydef 
  2 \left[\frac{M}{M-1}\right]^s +  
  \left[ \frac{4\ln(M-2)}{M} + \frac{\pi^2-6}{3}  \right]
  \left[ \frac{2}{M} + \frac{1}{2} \right]^{s-2}
\]
and considering $k \ge M$, we have that
\begin{eqnarray} \label{eq:upper_boundS1_tail}
\sum _{k_1+k_2=k} \frac {1}{\omega_{k_1} \omega_{k_2}} &\le& \frac{1}{k^s}  \left( 4 + \frac{2}{s-1} + \gamma \right) \label{eq:upper_boundS1_tail0} \\
&\le& \frac{1}{k^{s-1}} \left[ \frac{1}{M} \left( 4 + \frac{2}{s-1} + \gamma \right) \right] \label{eq:upper_boundS1_tail}
\end{eqnarray}
and
\begin{equation} \label{eq:upper_boundS2_tail}
\sum _{k_1+k_2=k} \frac {|k_1|}{\omega_{k_1} \omega_{k_2}} \le  \frac{1}{k^{s-1}} \left( 3 + \frac{2}{s-1} + \frac{\gamma}{2} \right)  .
\end{equation}
\end{lemma}
\begin{proof}
Let $k \ge M$. By Lemma~A.2 in \cite{BL}, we get
\begin{eqnarray*}
\phi_k &=& \sum_{k_1=1}^{k-1} \frac{1}{k_1^s(k-k_1)^s} \\
&\le& 
\frac{1}{k^s} \left( 2 \left[\frac{k}{k-1}\right]^s +  
\left[ \frac{4\ln(k-2)}{k} + \frac{\pi^2-6}{3}  \right]
\left[ \frac{2}{k} + \frac{1}{2} \right]^{s-2} \right) 
\\
& \le & \frac{1}{k^s} \gamma.
\end{eqnarray*}
The rest of the proof is a minor modification of the proof of Lemma~\ref{lem:estimates_sum_finite_case}.
\end{proof}

%%%% Bounds hat{C} %%%%
\begin{table}
\newcommand{\cocc}[3]{$\widehat{C}^{(#1,#2)}$ & $#3$ \\[1mm] \hline}
\renewcommand{\arraystretch}{1.7}

\centerline{\tiny \begin{tabular}{|r@{\hspace{2mm}}|@{\hspace{2mm}}l|}
\hline
%%%%%%%%%%%%%%%%%%%%%%%%%%%%%%%%%%%%%%%%%%%%%%%%%%%%%%%%%%
\cocc{1}{0}{ 
 \displaystyle{ \sum_{k_1=1}^{m-1}} \frac{4 \alpha_0}{2m-1}
(|\ba_{k_1}|+|\bb_{k_1}|) \left( 1+ \frac{1}{\left(1-\frac{k_1}{2m-1} \right)^s} \right)
}
\cocc{1}{1}{ 
\frac{2}{2m-1} + (2\alpha_0 |1+\ba_0| + 1) |\dL|  + \frac{4 |\ba_0+\alpha_0 \da_0|}{2m-1} +
\displaystyle{ \sum_{k_1=1}^{m-1}} \frac{4}{2m-1} ( |\ba_{k_1} + \alpha_0 \da_{k_1} |+|\bb_{k_1} + \alpha_0 \db_{k_1}  | ) \left( 1 +  \frac{1}{\left(1-\frac{k_1}{2m-1} \right)^s} \right)
}
\cocc{1}{2}{ 
\frac{4|\da_0| }{2m-1} + \displaystyle{ \sum_{k_1=1}^{m-1}} \frac{4}{2m-1} 
(|\da_{k_1}|+|\db_{k_1}|) \left( 1+\frac{1}{\left(1-\frac{k_1}{2m-1} \right)^s} \right)
}
\cocc{1}{3}{ 
0
}
\cocc{2}{0}{ 
\begin{array}{ll}
2 \alpha_0 (1+|\ba_0|+|1+\ba_0|) + \frac{8 \alpha_0}{2m-1} \left(4 + \frac{2}{s-1} + \gamma \right)
+  \displaystyle{ \sum_{k_1=1}^{m-1}} \frac{2 \alpha_0 k_1}{2m-1}
(|\ba_{k_1}|+|\bb_{k_1}|) \left( 1+ \frac{1}{\left(1-\frac{k_1}{2m-1} \right)^s} \right) \\
+  \displaystyle{ \sum_{k_1=1}^{m-1}} 2 \alpha_0
(|\ba_{k_1}|+|\bb_{k_1}|) \left( 1+ \frac{1}{\left(1-\frac{k_1}{2m-1} \right)^{s-1}} \right)
\end{array}
}
\cocc{2}{1}{
\begin{array}{ll}
2(1+|\ba_0+\alpha_0 \da_0|) + \frac{8}{2m-1} \left(4 + \frac{2}{s-1} + \gamma \right)
+  \displaystyle{ \sum_{k_1=1}^{m-1}} \frac{2 k_1}{2m-1}
( |\ba_{k_1} + \alpha_0 \da_{k_1} |+|\bb_{k_1} + \alpha_0 \db_{k_1}  | ) \left( 1+ \frac{1}{\left(1-\frac{k_1}{2m-1} \right)^s} \right)
\\
+  \displaystyle{ \sum_{k_1=1}^{m-1}} 2 
( |\ba_{k_1} + \alpha_0 \da_{k_1} |+|\bb_{k_1} + \alpha_0 \db_{k_1}  | ) \left( 1+ \frac{1}{\left(1-\frac{k_1}{2m-1} \right)^{s-1}} \right)
\end{array}
}
\cocc{2}{2}{ 
2 |\da_0| +  \displaystyle{ \sum_{k_1=1}^{m-1}} \frac{2 k_1}{2m-1}
(|\da_{k_1}|+|\db_{k_1}|) \left( 1+ \frac{1}{\left(1-\frac{k_1}{2m-1} \right)^s} \right)
+  \displaystyle{ \sum_{k_1=1}^{m-1}} 2 
(|\ba_{k_1}|+|\bb_{k_1}|) \left( 1+ \frac{1}{\left(1-\frac{k_1}{2m-1} \right)^{s-1}} \right)
}
\cocc{3}{0}{ 
4 \alpha_0 \left( 3 + \frac{2}{s-1}+\frac{\gamma}{2} \right)
}
\cocc{3}{1}{ 
12 + \frac{8}{s-1} + 2 \gamma
}
\end{tabular}}
\caption{The bounds $\hat{C}^{(l_1,l_2)}$.}
\label{t:bounds_hatC}
\end{table}
%%%%%%%%%%%%%%

The bounds (\ref{eq:upper_boundS1_tail}) and (\ref{eq:upper_boundS2_tail}) are used to find the $\hat{C}^{(l_1,l_2)}$ satisfying (\ref{eq:hatC}). %For further details, we refer to the program {\em hatC.mw}. 
The bounds  $\hat{C}^{(l_1,l_2)}$  are presented in Table~\ref{t:bounds_hatC}. We still need one last estimate before defining the bound $\hat{Z}_M(r,\dal)$.

\begin{lemma} \label{lm:Xi_M}
Let $\bL>0$, $\ba_0 \in \mathbb{R}$ and consider $m$ such that (\ref{eq:inv_condition2}) is satisfied. Define 
\[ \rho = \frac{M}{M \bL-\alpha_0|1+\ba_0|} >0 \]
and
\[
\Xi = \left( \begin{array}{cc} \frac{\rho^2}{M} \alpha_0 \left( |\ba_0|+|1+ \ba_0| \right) & \rho \\ \rho & \frac{\rho^2}{M} \alpha_0 \left( |\ba_0|+|1+ \ba_0| \right) \end{array} \right)
\]
Then for all $k \ge M$, $\Lambda_k$ is invertible and 
\begin{equation}
\left| {\Lambda_k}^{-1} \right| \lec \frac{1}{k} \Xi~ .
\end{equation}
\end{lemma}
\begin{proof} The fact that $\Lambda_k$ given by (\ref{eq:Lambda_k}) is invertible for all $k\ge M>m$ follows from the choice of $m$ given by (\ref{eq:inv_condition2}) and we then get that
\begin{eqnarray*}
{\Lambda_k}^{-1}
&=& \frac{1}{\tau_k^2+\delta_k^2} \left( \begin{array}{cc} \tau_k & -\delta_k \\ \delta_k  &  \tau_k \end{array} \right) .
\end{eqnarray*}
Since $k \ge M>\frac{\alpha_0 |1+ \ba_0|}{\bL}$,
\begin{eqnarray*}
|\delta_k|&=&k\bL-\alpha_0(1+\ba_0) \sin{k\bL} \\
&\ge& k\bL-\alpha_0|1+\ba_0| \\
&=& k \left( \bL - \frac{\alpha_0|1+\ba_0| }{k} \right) \\
&\ge& k \left( \bL - \frac{\alpha_0 |1+\ba_0| }{M} \right) = \frac{k}{\rho} >0 .
\end{eqnarray*}
Therefore, \[ \frac{1}{|\delta_k|} \le \frac{\rho}{k} \]
and then
\[  \left| \frac{\delta_k }{\tau_k^2+\delta_k^2} \right| \le  \frac{|\delta_k|}{\delta_k^2}
= \frac{1 }{|\delta_k|} \le \frac{1}{k} \rho . \] 
Finally, since $|\tau_k| \le \alpha_0 \left( |\ba_0|+|1+ \ba_0| \right) $, we get that 
\begin{eqnarray*}
\left| \frac{\tau_k }{\tau_k^2+\delta_k^2} \right| &\le&  \frac{\alpha_0 \left( |\ba_0|+|1+ \ba_0| \right)}{\tau_k^2+\delta_k^2} 
\le  \frac{\alpha_0 \left( |\ba_0|+|1+ \ba_0| \right)}{\delta_k^2} \\
 &\le& \frac{\rho^2 \alpha_0 \left( |\ba_0|+|1+ \ba_0| \right)}{k^2} 
\le \frac{1}{k} \left[ \frac{\rho^2 \alpha_0 \left( |\ba_0|+|1+ \ba_0| \right)}{M} \right].
\end{eqnarray*}
\end{proof}

We are now ready to define $\hat{Z}_M(r,\dal)$ in the fashion of Definition~\ref{def:radpol}. 

\begin{lemma}
Define
\begin{equation} \label{eq:hatZ}
\hat{Z}_M(r,\dal) \bydef \frac{1}{M^s} \left( \frac{\rho^2}{M} \alpha_0 \left( |\ba_0|+|1+ \ba_0| \right) + \rho \right) \left[ \sum_{l_1=1}^3 \sum_{l_2=0}^{4-l_1}\hat{C}^{(l_1,l_2)} r^{l_1} \dal^{l_2} \right] \left( \hspace{-.1cm} \begin{array}{cc} 1 \\ 1 \end{array}\hspace{-.1cm} \right)
\end{equation}
and consider $k\ge M$. Then
\[
\left| \left[ DT(x_\alpha+ru,\alpha)rv \right]_k \right| \lec \hat{Z}_M(r,\dal) \left( \frac{M}{k} \right)^s.
\]
\end{lemma}
\begin{proof} Let $k \ge M$. Combining equations (\ref{e:splitT}) and (\ref{eq:coeff_expansion}), and Lemma~\ref{lm:Xi_M}, we get that
\begin{eqnarray*}
\Bigl| \left[ DT(x_\alpha+ru,\alpha)rv \right]_k \Bigr| &=& \left| - \Lambda_k^{-1} \left[ Df(x_\alpha+ru,\alpha)rv - \dagA rv \right]_k \right| \\
&\lec& \sum_{l_1=1}^3 \sum_{l_2=0}^{4-l_1} |\Lambda_k^{-1}|  |c_k^{(l_1,l_2)}| r^{l_1} \dal^{l_2} \\ &\lec & \sum_{l_1=1}^3 \sum_{l_2=0}^{4-l_1} \frac{1}{k} \Xi  \frac{1}{k^{s-1}} \hat{C}^{(l_1,l_2)} \left( \hspace{-.1cm} \begin{array}{cc} 1 \\ 1 \end{array}\hspace{-.1cm} \right) r^{l_1} \dal^{l_2} \\
&=&\hat{Z}_M(r,\dal) \left( \frac{M}{k} \right)^s.
\end{eqnarray*}
\end{proof}

\begin{remark} \label{rem:pk_increasing}
Recalling the definitions of $Y_{_F}$, $Z_{_F}$ and $\hat{Z}_M$, given respectively by (\ref{eq:Y_F}), (\ref{eq:Z_F}) and (\ref{eq:hatZ}), one easily observe that the radii polynomials $p_k(r,\dal)$ from Definition~\ref{def:radpol} are monotone increasing in the variable $\dal \ge 0$.
\end{remark}

\subsection{First part of the proof of Theorem~\ref{thm:SOPS}: Rigorous computation of the branch \boldmath$\cF_0^*$\unboldmath ~using validated continuation} \label{sec:1st_part}

In Sections~\ref{sec:Y} and \ref{sec:Z}, we constructed the bounds $Y$ and $Z$, respectively. The coefficients in Tables~\ref{t:bounds_Dk}, \ref{t:boundsC} and \ref{t:bounds_hatC} provide us an analytical representation of the radii polynomials associated to (\ref{eq:f}). The following Procedure is an algorithm to compute a global continuous branch of solutions of $(\ref{eq:f})$.

%%%%%%%%%%%%%%
%%   Begin Procedure   %%
%%%%%%%%%%%%%%

\begin{procedure}\label{proc:branch}
To check the hypotheses of Lemma~\ref{lem:radpol} and Propostion~\ref{prop:continuum} on the interval $\alpha \in [\pi/2+\eps,2.3]$,
we proceed as follows.
\begin{enumerate} 
\item \label{i:initialize} 
Consider minimum and maximum step-sizes  $\Delta_{\min} =1 \times 10^{-15}$ and $\Delta_{\max}=2$, respectively. Initiate $s = 3 $, $m=6$, $M=2m-1$, $\alpha_0=\pi/2+\eps$, $r_0=0$, 
$\dal =5 \times 10^{-5} \in [\Delta_{\min},\Delta_{\max}]$, $\dal^0=0$,
and an approximate solution $\hat{x}_{_F}$ of $f^{(m)}(\cdot,\alpha_0)=0$ given in Figure~\ref{fig:x0}. Initiate $B_0=B_{\hat{x}_{_F}}(r_0)$.

\item \label{i:findx}
With a classical Newton iteration, find near $\hat{x}_{_F}$ an approximate solution $\bx_{_F}$ of
$f^{(m)}(x_{_F},\alpha_0)=0$. Calculate an approximate solution $\dx_{_F}$ of 
$D_xf^{(m)}(\bx_{_F},\alpha_0)\dx_{_F} + D_\alpha f^{(m)}(\bx_{_F},\alpha_0)=0$.
Using interval arithmetic, verify that conditions (\ref{eq:inv_condition1}) and (\ref{eq:inv_condition2}) are satisfied (this guarantees that the linear operator $A$ defined in (\ref{eq:A}) is invertible).

\item \label{i:coefficients}
Compute, using interval arithmetic, the coefficients of the radii polynomials
$p_k$, $k=0,\dots,M$ given in Definition~\ref{def:radpol}. This is the computationally most expensive step, since it involves computing all coefficients in Tables~\ref{t:bounds_Dk}, \ref{t:boundsC} and \ref{t:bounds_hatC}, and in particular requires the calculation of many loop terms.

\item  \label{i:continuum}
  Calculate numerically
  $\cI = [r_1^-,r_1^+] \bydef \bigcap_{k=0}^M  \{ r \ge 0 \, |\, p_k(r,0) \le 0 \}$.  Consider
  $B_1^- \bydef B_{\bx_{_F}}(r_1^-)$ and $B_1^+ \bydef B_{\bx_{_F}}(r_1^+)$. Verify that $B_0 \subset B_1^+$ or $B_1^- \subset B_0$.

\item \label{i:calcI}
  Calculate numerically
  $I = [I_-,I_+] \bydef \bigcap_{k=0}^M  \{ r \ge 0 \, |\, p_k(r,\dal) \le 0 \}$.
  \vspace*{-1.5mm}
  \begin{itemize}
    \item
      If $I = \varnothing$ then go to Step~\ref{i:failure}. 
    \item
      If $I \neq \varnothing$ then let $r = \frac{I_-+I_+}{2} $.  
      Compute with interval arithmetic $p_k(r,\dal)$.
      If $p_k(r,\dal) <0$ for all $k=0,\dots,M$ then go to Step~\ref{i:success};
      else go to Step~\ref{i:failure}.
   \end{itemize}

\item\label{i:success} %(success)
  Update $\dal^0 \leftarrow \dal$ and $r_0 \leftarrow r$. 
  If $\frac{10}{9}\dal \leq \Delta_{\max}$ then
  update $\dal \leftarrow \frac{10}{9} \dal$ and go to
  Step~\ref{i:calcI}; else go to Step~\ref{i:finalize}.

\item\label{i:failure} %(failure)
  If $\dal^0 >0$ then go to Step~\ref{i:finalize};
  else if $\frac{9}{10}\dal \geq \Delta_{\min}$ then
  update $\dal \leftarrow \frac{9}{10} \dal$ and go to Step~\ref{i:calcI}; 
  else go to Step~\ref{i:failedstep}.

\item\label{i:finalize}
  The continuation step has succeeded.
  Store, for future reference, $\bx_{_F}$, $\dx_{_F}$, $r_0$, $\alpha_0$ and $\dal^0$.
  Determine $\alpha_1$ approximately equal to, but interval arithmetically less than, $\alpha_0 + \dal^0$. 
  Make the updates $\alpha_0 \leftarrow \alpha_1$, $\dal \leftarrow \dal^0$,
  $\hat{x}_{_F} \leftarrow \bx_{_F} + \dal^0 \dx_{_F}$ and $\dal^0 \leftarrow 0$.
  If one of the last two components of $\hat{x}_{_F}$ has magnitude larger than $1 \times 10^{-9}$, update  
      $\hat{x}_{_F} \leftarrow (\hat{x}_{_F},0,0)$, $m \leftarrow m+1$ and $M \leftarrow 2m-1$. Update $B_0 \leftarrow B_{\hat{x}_{_F}}(r_0)$ and go to Step~\ref{i:findx} for the next continuation step.

\item\label{i:failedstep}
  The continuation step has failed.
  Either decrease~$\Delta_{\min}$ and return to Step~\ref{i:failure}; 
  or increase $M$ and return to Step~\ref{i:coefficients};
  or increase $m$ and return to Step~\ref{i:findx}.
  Alternatively, terminate the procedure \emph{unsuccessfully} at $\alpha=\alpha_0$
  (although with success on $[\pi/2+\eps,\alpha_0]$).

\end{enumerate}
\end{procedure}

%%%%%%%%%%%%%
%%   End Procedure   %%
%%%%%%%%%%%%%

\begin{figure}[ht!]
\footnotesize\centering
\begin{tabular}{|c|c|c|c|c|c|c|}
\hline
$\bL$ & $1.570599180042083$  \\
\hline
$\ba_0$  & 0  \\
\hline
$\ba_1$ & $0.000393777377493$  \\
\hline
$\bb_1$ & $0.031377227341359$  \\
\hline
$\ba_2$ & $-0.000389051487791$  \\
\hline
$\bb_2$ & $0.000206800585095$  \\
\hline
$\ba_3$ & $-0.000004694294098$  \\
\hline
$\bb_3$ & $-0.000001372932742$  \\
\hline
$\ba_4$ & $-0.000000031481138$  \\
\hline
$\bb_4$ & $-0.000000035052666$  \\
\hline
$\ba_5$ & $-0.000000000114467$  \\
\hline
$\bb_5$ & $-0.000000000397361$  \\
\hline
\end{tabular}
\caption{Approximate zero $\hat{x}_{_F}$ at the parameter value $\alpha_0=\frac{\pi}{2}+\eps$.}
\label{fig:x0}
\end{figure}

The {\it Matlab} program {\em intvalWrightCont.m}, which can be found at \cite{code}, performs Procedure~\ref{proc:branch} successfully on the parameter interval $[\pi/2+\eps,2.3]$. Hence, by construction, we get the existence of a continuous one dimensional branch of periodic solutions $\cF_0^*$ which does not have any fold in the range of parameter $[\pi/2+\eps,2.3]$. This result follows from the uniform contraction principle and Proposition~\ref{prop:continuum}. The last step of the proof is to show that $\cF_0^*$ is the branch of SOPS of Wright's equation that bifurcates from the trivial solution at $\alpha=\pi/2$.

\subsection{Second part of the proof of Theorem~\ref{thm:SOPS}: Bifurcation analysis at \boldmath$\alpha=\pi/2$\unboldmath ~ to show that \boldmath$\cF^*_0 \subset \cF_0$\unboldmath} \label{sec:2nd_part}

In this section, we show that the branch $\cF_0^*$ comes from the Hopf bifurcation at $\alpha=\pi/2$. For a detailed analysis of this Hopf bifurcation, we refer to Section~11.4 of \cite{Hale-Lunel}.  Consider the change of variable $y(t)=\beta z(t)$. Plugging $y(t)=\beta z(t)$ in Wright's equation (\ref{eq:wright}), we get 
\begin{equation} \label{eq:new_wright}
\dot{z}(t)=-\alpha z(t-1) [1+\beta z(t)].
\end{equation}
Consider the problem of looking for periodic solutions of (\ref{eq:new_wright}), with the parameter now being $\beta \ge 0$ ($\alpha$ is now considered as a variable). We impose to the periodic solutions the conditions $z(0)=0$ and $\dot{z}(0)=-1$. More precisely, we consider the problem
\begin{equation} \label{eq:new_wright1}
\left\{ 
\begin{array}{llll}
\dot{z}(t)=-\alpha z(t-1) [1+\beta z(t)], ~\beta \ge 0, \\
z\left( t+\frac{2 \pi}{L} \right)=z(t), \\
z(0)=0, ~~ \dot{z}(0)=-1.
\end{array}
\right.
\end{equation}
When $\beta=0$, $\alpha=\pi/2$ and $L=\pi/2$, equation (\ref{eq:new_wright1}) has solution $z(t)=-\frac{2}{\pi} \sin \left( \frac{\pi}{2}t \right)$. This solution corresponds to the Hopf bifurcation point $y(t)= 0  \left( -\frac{2}{\pi} \sin \left( \frac{\pi}{2}t \right) \right) =0$, when $\alpha=\pi/2$ and $L=\pi/2$. The idea is to use validated continuation (in the parameter $\beta \ge 0$) on problem (\ref{eq:new_wright1}) and to connect the rigorously computed branch of SOPS of (\ref{eq:new_wright1}) to the left point of $\cF_0^*$. It is important to note that this new validated continuation cannot help ruling out the existence of fold in the space $(\alpha,y)$, but only in the space $(\beta,z)$. 

Considering the periodic solution $z(t)$ in Fourier expansion, we do as in Section~\ref{sec:f_set_up} and consider a function to solve for.  Defining $X=(\alpha,x)$, it can be shown that an equivalent problem of (\ref{eq:new_wright1}) is $F(X,\beta)=0$, where 
\begin{equation} \label{eq:F}
F_k(X,\beta)= \left\{
\begin{array}{ccc} 
-1+2L \sum_{k=1}^\infty k b_k ,~k=-1 \\
\begin{array}{cc} a_0+2 \sum_{k=1}^{\infty} a_k \\ \alpha \left( a_0 + \beta a_0^2 + 2 \beta \sum_{k_1=1}^{\infty} (\cos{k_1L})\left( a_{k_1}^2+b_{k_1}^2 \right) \right) \end{array}  ,~k=0 \\
R_k(L,\alpha) \left( \begin{array}{cc} a_k \\ b_k\end{array} \right) + \alpha \beta \displaystyle{\sum_{\stackrel{k_1+k_2=k}{k_i \in \mathbb{Z}}}} \Theta_{k_1}(L) \left(
\begin{array}{cc} a_{k_1}a_{k_2}-b_{k_1}b_{k_2} \\ a_{k_1}
b_{k_2}+b_{k_1}a_{k_2} \end{array} \right)  ,~k \ge1,
\end{array} \right. 
\end{equation}
where
\[
R_k(L,\alpha) \bydef  \left( \begin{array}{cc} \alpha \cos{kL} & -kL+\alpha \sin{kL} \\ kL-\alpha \sin{kL}
& \alpha \cos{kL} \end{array} \right)
\]
and 
\[ \Theta_{k_1}(L) \bydef \left( \begin{array}{cc} \cos{k_1 L} & \sin{ k_1
L} \\ -\sin{k_1 L} & \cos{k_1 L} \end{array} \right). \]
To apply validated continuation on problem (\ref{eq:F}), with $\beta \ge 0$ being the parameter, we need to construct the radii polynomials. Here, we do not provide analytically the coefficients of the radii polynomials associated to (\ref{eq:F}), since they are similar to the ones associated to (\ref{eq:f}). A procedure similar to Procedure~\ref{proc:branch} is applied on $(\ref{eq:F})$ to get the existence of a continuous branch of SOPS of (\ref{eq:new_wright1}) on the parameter range $\beta \in [0,\beta_0]$, where $\beta_0 \bydef 0.099847913753516$. We denote this branch by $\mathcal{G}_0^*$. See Figure~\ref{fig:bif_image} for a geometric representation of $\mathcal{G}_0^*$. At the right most point of $\mathcal{G}_0^*$, we have a set $B_0^*$ containing a unique solution of $F(X,\beta_0)=0$. Using a similar argument than the one presented in Proposition~\ref{prop:continuum}, we can show, via the change of coordinates $y= \beta z$, that the solution in the set $B_0^*$ and the solution on the left most part of the branch $\cF_0^*$ are the same. Hence, we proved that $\cF_0^* \subset \cF_0$. \qed

\begin{figure} 
\centerline{\includegraphics[width=6cm]{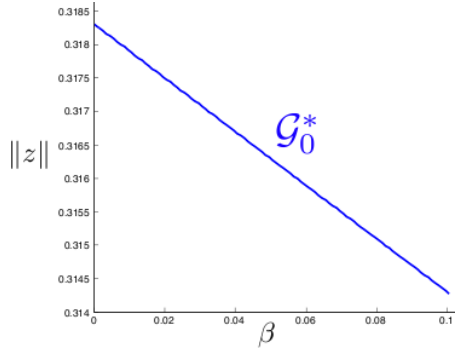}} 
\caption{A branch of SOPS of (\ref{eq:new_wright1}) on $[0,\beta_0]$.}
\label{fig:bif_image}
\end{figure}

\section{Future Work and Acknowledgments} \label{sec:conclusion}

As mentioned in Section~\ref{sec:introduction}, we believe that Theorem~\ref{thm:SOPS} could be improved significantly. The reason why the proof was stopped at $\alpha=2.3$ is due to the fact that the  {\it Matlab} program {\em intvalWrightCont.m} \cite{code} becomes slow for large $\alpha$. Indeed, the evaluation of the coefficients of the radii polynomials is computationally expensive, mainly because of all the iterative loop evaluations in Step~\ref{i:coefficients} of Procedure~\ref{proc:branch}, a task that the interval arithmetic {\em Intval} is not efficient at doing. Using a different programming language (like $C$ or $C++$) would decrease significantly the computational time. We believe that we could push the parameter value up to $\alpha=3$ using a $C$ program. This speculation is based on simulations that were done in {\em Matlab} without interval arithmetic. We could, with the new program, reduce also the value of $\eps$ significantly. 

It worths mentioning that validated continuation can be applied to other delay equations. In particular, one interesting future project would be to apply the method to study periodic solutions of the Mackey-Glass equation (see \cite{mackey-glass})
\begin{equation} \label{eq:mg}
\dot{x}(t)=\frac{\alpha x(t-\tau)}{1+[x(t-\tau)]^n} -\beta x(t), ~~~~\alpha, \beta, \tau>0, n \in \mathbb{N},
\end{equation}
for which the existence of more than one SOPS in (\ref{eq:mg}) is an open conjecture, for certain range of parameters. We refer to \cite{liz-rost} for more details on this conjecture.

The author would like to thank to Roger Nussbaum, John Mallet-Paret, Konstantin Mischaikow and Eduardo Liz for helpful discussions. Also, the author would like to give a special thank to Jan Bouwe van den Berg for his idea about the formulation of the bifurcation analysis presented in Section~\ref{sec:2nd_part}.


\begin{thebibliography}{10}

\bibitem{BL}
{\sc Jan Bouwe van den Berg and Jean-Philippe Lessard.}
\newblock Chaotic braided solutions via rigorous numerics: chaos in the
  swift-hohenberg equation.
\newblock {\em SIAM Journal on Applied Dynamical Systems}, 7(3):988--1031, 2008.

\bibitem{BLM}
{\sc  Jan Bouwe van den Berg, Jean-Philippe Lessard and Konstantin Mischaikow.}
\newblock Global smooth solution curves using rigorous branch following.
\newblock To appear in {\em Mathematics of Computation}, 2009.

\bibitem{CH}
{\sc Shui Nee Chow and Jack K. Hale.}
\newblock {\em Methods of bifurcation theory}, volume 251 of {\em Grundlehren
  der Mathematischen Wissenschaften [Fundamental Principles of Mathematical
  Science]}.
\newblock Springer-Verlag, New York, 1982.

\bibitem{C-MP1}
{\sc  Shui-Nee Chow and John Mallet-Paret.}
\newblock Integral averaging and bifurcation.
\newblock {\em Journal of Differential Equations}, 26(1):112--159, 1977.

\bibitem{DLM}
{\sc Sarah Day, Jean-Philippe Lessard and Konstantin Mischaikow.}
\newblock Validated continuation for equilibria of {PDEs}.
\newblock {\em SIAM Journal on Numerical Analysis}, 45(4):1398--1424, 2007.

\bibitem{GL}
{\sc Marcio Gameiro and Jean-Philippe Lessard.}
\newblock A priori estimates and validated continuation for equilibria of high dimensional {PDEs}.
\newblock {\em Preprint}, 2009.

\bibitem{GLM}
{\sc Marcio Gameiro, Jean-Philippe Lessard and Konstantin Mischaikow.}
\newblock Validated continuation over large parameter ranges for equilibria of
  {PDEs}.
\newblock {\em Mathematics and computers in simulation}, 79(4): 1368-1382, 2008.

\bibitem{Hale-Lunel}
{\sc Jack  K. Hale and Sjoerd M. Verduyn Lunel.}
\newblock Introduction to Functional Differential Equations.
\newblock {\em Springer, 1993}

\bibitem{code}
http://www.math.rutgers.edu/$\sim$lessard/Wright

\bibitem{Jones2}
{\sc Stephen~G. Jones.}
\newblock The existence of periodic solutions of $f\sp{\prime} (x)=-\alpha
  f(x-1)\{1+f(x)\}$.
\newblock {\em Journal of Mathematical Analysis and Applications}, 5:435--450,
  1962.

\bibitem{Jones1}
{\sc Stephen~G. Jones.}
\newblock On the nonlinear differential-difference equation $f\sp{\prime}
  (x)=-\alpha f(x-1)\{1+f(x)\}$.
\newblock {\em Journal of Mathematical Analysis and Applications}, 4:440--469,
  1962.

\bibitem{kakutani-markus}
{\sc Shizuo~Kakutani and Lawrence~Markus.}
\newblock On the non-linear difference-differential equation $y\sp{\prime}
  (t)=[a-by(t-\tau )]y(t)$.
\newblock {\em Contributions to the theory of nonlinear oscillations, Princeton
  University Press, Princeton, N.J.}, IV(41):1--18, 1958.

\bibitem{kaplan-yorke1}
{\sc James~A. Kaplan and James~A. Yorke.}
\newblock On the stability of a periodic solution of a differential delay
  equation.
\newblock {\em SIAM Journal on Mathematical Analysis}, 6:268--282, 1975.

\bibitem{kaplan-yorke2}
{\sc James~A. Kaplan and James~A. Yorke.}
\newblock On the nonlinear differential delay equation $x'(t)=-f(x(t),$
  $x(t-1))$.
\newblock {\em Journal of Differential Equations}, 23(2):293--314, 1977.

\bibitem{liz-rost}
{\sc Eduardo Liz and Gergely R\"ost.} 
\newblock Dichotomy results for delay differential equations with negative Schwarzian.
\newblock {\em Preprint}.

\bibitem{MP_Walther}
{\sc John Mallet-Paret and Hans-Otto Walther}.
\newblock Rapid oscillations are rare in scalar systems governed by monotone negative feedback with a time delay.
\newblock {\em Preprint}, Math. Inst., University of Giessen, 1994.

\bibitem{mackey-glass}
{\sc Michael C. Mackey and Leon Glass.}
\newblock Oscillations and chaos in physiological control system.
\newblock {\em Science}, 197: 287-289, 1977.

\bibitem{nussbaum3}
{\sc Roger Nussbaum.}
\newblock Periodic solutions of analytic functional differential equations are
  analytic.
\newblock {\em Michigan Math. J.}, 20:249--255, 1973.

\bibitem{nussbaum4}
{\sc Roger Nussbaum.}
\newblock The range of periods of periodic solutions of $x'(t)=-\alpha
  f(x(t-1))$.
\newblock {\em Journal of Mathematical Analysis and Applications},
  58(2):280--292, 1977.

\bibitem{nussbaum2}
{\sc Roger Nussbaum.}
\newblock Asymptotic analysis of some functional-differential equations.
\newblock {\em Dynamical systems, II}, pages 277--301, 1982.

\bibitem{nussbaum5}
{\sc Roger Nussbaum.}
\newblock Wright's equation has no solutions of period four.
\newblock {\em Proceedings of the Royal Society of Edinburgh. Section A.
  Mathematics}, 113(3-4):281--288, 1989.

\bibitem{regala}
{\sc Benjamin T. Regala.}
\newblock Periodic solutions and stable manifolds of generic delay
  differential equations.
\newblock {\em PhD thesis}, Division of Applied Mathematics, Brown University, 1989.

\bibitem{Intlab}
{\sc Siegfried M. Rump}.
\newblock INTLAB - INTerval LABoratory. Version 5.5.
\newblock Available at www.ti3.tu-harburg.de/rump/intlab/.

\bibitem{wright}
{\sc Edward~M. Wright.}
\newblock A non-linear difference-differential equation.
\newblock {\em Journal f\"{u}r die reine und angewandte Mathematik},
  194:66--87, 1955.

\bibitem{xie_thesis}
{\sc Xianwen~Xie.}
\newblock {\em Uniqueness and stability of slowly oscillating periodic
  solutions of differential delay equations}.
\newblock PhD thesis, Rutgers University, 1991.

\bibitem{xie1}
{\sc Xianwen~Xie.}
\newblock Uniqueness and stability of slowly oscillating periodic solutions of
  delay equations with unbounded nonlinearity.
\newblock {\em Journal of Differential Equations}, 103(2):350--374, 1993.

\end{thebibliography}
\end{document}